\documentclass[11pt,a4paper]{amsart}
\usepackage{amssymb}
\usepackage{mathrsfs}
\usepackage{microtype}

\usepackage{xcolor}
\definecolor{linkblue}{RGB}{38,68,105}
\definecolor{citegreen}{RGB}{45,88,68}

\usepackage[colorlinks=true,linkcolor=linkblue,citecolor=citegreen,urlcolor=linkblue,pdfborder={0 0 0}]{hyperref}

\usepackage{tikz}
\usetikzlibrary{arrows.meta,calc}

\newtheorem{theorem}{Theorem}[section]
\newtheorem{proposition}[theorem]{Proposition}
\newtheorem{corollary}[theorem]{Corollary}
\newtheorem{lemma}[theorem]{Lemma}
\theoremstyle{definition}

\newtheorem{conjecture}[theorem]{Conjecture}
\newtheorem{definition}[theorem]{Definition}
\theoremstyle{remark}
\newtheorem{remark}[theorem]{Remark}
\numberwithin{equation}{section}

\newcommand{\C}{\mathbb C}
\newcommand{\R}{\mathbb R}
\newcommand{\PP}{\mathbb P}
\newcommand{\HH}{\mathbb H}
\newcommand{\A}{\mathcal A}
\newcommand{\X}{\mathcal X}
\newcommand{\K}{\mathcal K}
\DeclareMathOperator{\SL}{SL}
\DeclareMathOperator{\SU}{SU}
\DeclareMathOperator{\PSL}{PSL}
\DeclareMathOperator{\PSU}{PSU}
\DeclareMathOperator{\Hol}{Hol}
\DeclareMathOperator{\Hom}{Hom}
\DeclareMathOperator{\tr}{tr}
\DeclareMathOperator{\Fix}{Fix}
\DeclareMathOperator{\arsinh}{arsinh}
\DeclareMathOperator{\Sph}{Sph}
\DeclareMathOperator{\Pic}{Pic}
\DeclareMathOperator{\divisor}{div}
\DeclareMathOperator{\Gr}{Gr}

\newcommand{\Z}{\mathbb Z}
\newcommand{\norm}[1]{\lVert #1\rVert}

\title[Finiteness of Fuchsian equations with unitarizable monodromy]{Finiteness of Fuchsian equations with unitarizable monodromy on compact Riemann surfaces}

\author{Sheng Tan}

\address{School of Mathematical Sciences, Capital Normal University, Beijing 100048, China} 
\email{tansheng2018@outlook.com, tansheng@cnu.edu.cn}

\subjclass[2020]{Primary 34M35; Secondary 30F30, 30F45, 20E08}
\keywords{Fuchsian equations, accessory parameters, unitarizable monodromy, spherical conical metrics, projective connections, character varieties, real trees}
\date{}

\thanks{This work is supported by NSFC (Grant No.~12601116) and BJNSF (Grant No.~1264052).}

\begin{document}

\begin{abstract}
We prove that, on any compact Riemann surface, there are only finitely many meromorphic projective connections with prescribed singular points and real exponent differences whose projective monodromy is unitarizable. We also establish properness and growth estimates for holonomy, including a relative version for smooth marked families, and derive finiteness results for spherical conical metrics and ramification divisors of stable extensions.
\end{abstract}

\maketitle



\section{Introduction}\label{sec:introduction}

The monodromy of a Fuchsian differential equation depends not only on its singular points and local exponents, but also on its accessory parameters. For second-order equations, one can ask whether there are only finitely many choices of these parameters for which the projective monodromy is unitarizable. This question is closely related to spherical metrics with conical singularities \cite{Ere20,Ere21,Ere24}.

Let $X$ be a compact connected Riemann surface of genus $g$, let $D=\{a_1,\ldots,a_n\}$ be a finite set of distinct points, and let $\boldsymbol\beta=(\beta_1,\ldots,\beta_n)\in\R^n$. We consider equations in projective normal form
\begin{equation}\label{eq:intro-equation}
 y''+Q_z y=0.
\end{equation}
The local coefficients $Q_z$ form a meromorphic projective connection. The prescribed exponent differences mean that, in a coordinate centered at $a_i$,
\begin{equation}\label{eq:intro-principal}
 Q_z(z)=\frac{1-\beta_i^2}{4z^2}+O(z^{-1}).
\end{equation}
The projective monodromy is \emph{unitarizable} if it is conjugate to a subgroup of $\PSU(2)$. Throughout the paper, we fix the complex structure of $X$, the marked points $D$, and the exponent differences.

Eremenko formulated the following conjecture for arbitrary compact Riemann surfaces \cite[Section~2]{Ere20}. In a later problem note, he also raised the questions of discreteness and finiteness for the sphere \cite{Ere21}.

\begin{conjecture}[{Eremenko, \cite{Ere20}}]\label{question:eremenko}
On any compact Riemann surface, the set of Fuchsian equations with prescribed singularities and exponents, and with unitarizable monodromy is finite.
\end{conjecture}

For three marked points on the sphere the conjecture is trivial. For four points, Eremenko proved finiteness by passing to the elliptic form of Heun's equation \cite[Theorem~1]{Ere20}. His proof is based on asymptotics of monodromy along two elliptic periods together with an argument for discreteness using entire functions of one complex variable. However, this argument does not extend to equations with more than four singular points. Our first main result resolves the conjecture in the stated generality.

\begin{theorem}\label{thm:unitary}
Let $X$ be a compact connected Riemann surface, let $D=\{a_1,\ldots,a_n\}\subset X$ be a finite set of distinct points, and let $\beta_1,\ldots,\beta_n\in\R$. There are only finitely many meromorphic projective connections that are holomorphic outside $D$, have poles of order at most two along $D$, have coefficients $(1-\beta_i^2)/4$ of the $z^{-2}$ terms at $a_i$, and have unitarizable projective monodromy.
\end{theorem}

The geometry of holonomy provides a stronger formulation. Let $\boldsymbol\delta=(\delta_1,\ldots,\delta_n)$ be an arbitrary tuple of complex numbers, and let $\A_X=\A(X,D,\boldsymbol\delta)$ denote the space of meromorphic projective connections on $X$ that are holomorphic outside $D$ and have coefficient $\delta_i$ of the $z^{-2}$ term at each $a_i$. When nonempty, it is an affine space modeled on
\[
 V=H^0(X,K_X^{\otimes2}(D)),
\]
where $K_X$ is the canonical bundle. Fix a basepoint $x_0\in X\setminus D$, denote $\Gamma=\pi_1(X\setminus D,x_0)$, and fix a theta characteristic on $X$, namely a holomorphic line bundle whose square is $K_X$. One can define a  holonomy map
\[
 \Hol_X:\A_X\longrightarrow\X(\Gamma),
\]
by a basis of local solutions normalized to have Wronskian one, where $\X(\Gamma)$ is the affine $\SL_2(\C)$ character variety. Let $\K$ be the image of $\Hom(\Gamma,\SU(2))$ in this variety. Let $\overline X$ be the conjugate Riemann surface of $X$, equipped with the conjugate marked data and theta characteristic. We take the identity map on the underlying topological surface as the marking and write $\A_{\overline X}$ for the corresponding accessory space. Our second main result is a stronger statement about holonomy. 

\begin{theorem}\label{thm:holonomy}
Suppose $\A_X$ is nonempty. For arbitrary fixed complex coefficients $\boldsymbol\delta$, the map $\Hol_X$ is proper, and the complex analytic relation
\begin{equation}\label{eq:intro-relation}
 Z=\{(Q,R)\in\A_X\times\A_{\overline X}:\Hol_X(Q)=\Hol_{\overline X}(R)\}
\end{equation}
is closed and discrete. Moreover, $\Hol_X^{-1}(\K)$ is finite.
\end{theorem}

In particular, the last conclusion also applies to connections with nonsemisimple monodromy whose semisimplifications are unitary. It also follows from properness that every complex character has finitely many realizations and that holonomy is finite onto its reduced closed analytic image. Moreover, the images associated with $X$ and $\overline X$ have a closed discrete intersection; see Corollaries~\ref{cor:finite-holonomy} and~\ref{cor:image-intersection}. We also establish relative properness of holonomy for local holomorphic families of smooth marked curves with disjoint marked sections, fixed local coefficients, and compatible affine families of accessory spaces. On compact subsets of the base, the logarithmic size of finitely many monodromy traces grows like the square root of the accessory norm. These estimates, together with restricted real analyticity, give a uniform finite bound for the number of connections with unitary characters over each such compact subset; see Proposition~\ref{prop:relative-properness} and Corollary~\ref{cor:uniform-count}.

There are several earlier works relevant to this approach. Culler--Shalen and Morgan--Shalen related degenerations of representations to actions on trees \cite{CS83,MS84}. Such limiting actions can also be obtained analytically through equivariant harmonic maps \cite{DDW98}. For a fixed closed curve of genus at least two, Gallo--Kapovich--Marden proved that holonomy is a proper embedding \cite{GKM00}. Dumas used trees of limiting quadratic differentials to obtain two-sided estimates for holonomy growth and to reconstruct the limiting differential in the abelian case as the square of a holomorphic one-form \cite{Dum17}. He also used limiting directions to study intersections in character varieties \cite{Dum15}. Our argument establishes the corresponding comparisons for regular-singular equations, where the limiting quadratic differential initially has at most simple poles.

Discreteness of holonomy intersections is related to real projective structures. Faltings studied these structures with prescribed cusp and orbifold data \cite{Fal83}, and Luo proved local results for monodromy maps on punctured surfaces \cite{Luo93}. Baba proved discreteness of intersections of Poincar\'e holonomy varieties for distinct marked closed conformal structures, including opposite orientations \cite{Bab23}. Recent work of Dimakis--Wentworth concerns discreteness of real opers for special linear groups in the moduli space of irreducible flat connections \cite{DW26}.

We next consider spherical metrics with conical singularities. If $f$ is a developing map, the corresponding metric is
\begin{equation}\label{eq:spherical-metric}
 ds^2=\frac{4|f'|^2}{(1+|f|^2)^2}|dz|^2,
\end{equation}
and positive exponent differences $\beta_i$ correspond to cone angles $2\pi\beta_i$. Two metrics are projectively equivalent if their developing maps $f_1$ and $f_2$ satisfy $f_2=M\circ f_1$ for a M\"obius transformation $M$ of the target sphere. By Theorem~\ref{thm:unitary}, there are only finitely many projective equivalence classes for any prescribed tuple of positive cone angles. The description of a projective class in terms of invariant Hermitian forms \cite{Ere20,Ere24} shows that every noncoaxial class contains exactly one metric. We establish an obstruction to coaxial monodromy in terms of divisors on arbitrary compact Riemann surfaces and deduce finiteness of spherical metrics for positive nonintegral angles when $2g-2+n>0$.

Existence and uniqueness of these metrics have been studied by analytic and geometric methods. Troyanov developed existence criteria for prescribed curvature and conical singularities \cite{Tro91}; Luo--Tian proved uniqueness on the sphere with at least three cone points and all angles less than $2\pi$ \cite{LT92}, with another proof provided by the singular sphere-covering inequality \cite{BGJ+19}. Mondello--Panov obtained angle constraints and properness under a nonbubbling condition \cite{MP16,MP19}. Mazzeo--Zhu developed the deformation theory for larger angles \cite{MZ22}, and Huang--Lin related unitarizable Fuchsian equations to degree formulas for curvature equations \cite{HL25}. Reducible metrics and their character differentials were studied by Chen--Wang--Wu--Xu and Wei--Wu--Xu \cite{CWW+15,WWX26}.

For integral angles, the applications connect with stable vector bundles. Li--Song--Xu constructed a correspondence between noncoaxial metrics and stable extensions of line bundles, and proved finiteness for almost every effective divisor of odd degree greater than $2g-2$ \cite{LSX21}. We obtain finiteness for every such divisor. Existence is provided by Troyanov and Bartolucci--De Marchis--Malchiodi \cite{Tro91,BDM+11,DM10}; the elliptic case, including an algebraic counting bound, is already covered by Wang \cite{Wang21}. Using the result of Feng--Song--Xu on injectivity for fixed line bundles \cite{FSX26}, we also prove that the ramification divisor map, which is real analytic, has finite fibers and that its differential has maximal real rank on an open dense subset. 

The paper is organized as follows. In Section~\ref{sec:holonomy} we introduce projective connections, accessory spaces, and holonomy. Section~\ref{sec:trees} develops the local estimates and tree limits, proves properness and trace growth, and treats smooth marked families. In Section~\ref{sec:comparison} we compare limiting quadratic differentials, using tree rigidity and harmonic periods. Section~\ref{sec:finiteness} proves the finiteness theorems and establishes uniform counting in smooth marked families. Section~\ref{sec:geometric} presents the applications to spherical conical metrics, stable extensions, and meromorphic functions.


\section{Projective connections and holonomy}\label{sec:holonomy}

In this section, we introduce projective connections, accessory spaces, holonomy maps, and the associated trace functions, following \cite{Luo93,Dum17,Gol09}.

 
\subsection{Accessory spaces}\label{subsec:accessory}

Let $X$ be a compact connected Riemann surface of genus $g$, let $D=\{a_1,\ldots,a_n\}$ be a finite marked set, and let $K_X$ denote the canonical line bundle. We also use $D$ for the reduced divisor $\sum_i a_i$. For each $a_i\in D$, let $\delta_i$ be the prescribed coefficient of $z^{-2}$ in the local expansion of $Q_z$, and set $\boldsymbol\delta=(\delta_1,\ldots,\delta_n)\in\C^n$.

\begin{definition}[Projective connections and accessory spaces]\label{def:projective}
A meromorphic projective connection on $X$ is a collection of meromorphic local coefficients $Q_z$ satisfying
\begin{equation}\label{eq:projective-change}
 Q_w(w)=Q_z(z(w))\left(\frac{dz}{dw}\right)^2+\frac12\mathcal S(z)(w),
\end{equation}
where $\mathcal S(f)=f^{(3)}/f'-\tfrac32(f''/f')^2$ is the Schwarzian derivative. The \emph{accessory space} $\A_X=\A(X,D,\boldsymbol\delta)$ consists of the projective connections holomorphic outside $D$ with local expansion
\[
 Q_z(z)=\delta_i z^{-2}+O(z^{-1})
\]
in a coordinate centered at $a_i$. Affine coordinates on this space are called \emph{accessory parameters}.
\end{definition}

The coefficient $\delta_i$ is independent of the centered coordinate: if $z=cw+O(w^2)$, then $z^{-2}(dz/dw)^2=w^{-2}+O(w^{-1})$, and the Schwarzian term is holomorphic. The difference of two elements of $\A_X$ is a quadratic differential with at most simple poles along $D$ \cite{Luo93,Dum17}. Thus, if $\A_X$ is nonempty, choosing $Q_0\in\A_X$ gives
\begin{equation}\label{eq:affine-space}
 \A_X=Q_0+V,\quad V=H^0(X,K_X^{\otimes2}(D)).
\end{equation}
We write $m=\dim_\C V$ for the dimension of $\A_X$.

We first discuss existence and the cases with at most two marked points on the sphere. Every compact Riemann surface has a holomorphic projective connection, obtained from its uniformizing projective charts. For $n>0$, consider the exact sequence
\[
 0\longrightarrow K_X^{\otimes2}(D)\longrightarrow K_X^{\otimes2}(2D)
 \longrightarrow\bigoplus_{i=1}^n\C_{a_i}\longrightarrow0,
\]
where the last map extracts the coefficients of the $z^{-2}$ terms. By Serre duality,
\[
 H^1(X,K_X^{\otimes2}(D))\simeq H^0(X,K_X^{-1}(-D))^*.
\]
The right side vanishes for $g\ge2$, for $g=1,n\ge1$, and for $g=0,n\ge3$. Hence any prescribed coefficients are realized in these cases. Combining Serre duality with Riemann--Roch \cite{For81}, we obtain
\begin{equation}\label{eq:dimension}
 m=\begin{cases}
  \max\{n-3,0\},&g=0,\\
  n,&g=1,\ n\ge1,\\
  1,&g=1,\ n=0,\\
  3g-3+n,&g\ge2.
 \end{cases}
\end{equation}
For the sphere with at most two marked points, the conditions for nonemptiness can be stated explicitly. With no marked points, the standard projective connection is the only one. If only infinity is marked, the prescribed coefficient must be zero. If $0$ and infinity are marked, the two prescribed coefficients must agree; in that case, the connection is $Q_z=\delta_1/z^2$. Thus, whenever the accessory space is nonempty, it consists of a single point.

For $n\ge3$, we can describe the accessory space explicitly. After a M\"obius change of coordinate, we may assume $a_n=\infty$ and set $r_0=n-1$. Then
\begin{equation}\label{eq:partial-fractions}
 Q(z)=\sum_{i=1}^{r_0}\left(\frac{\delta_i}{(z-a_i)^2}+\frac{c_i}{z-a_i}\right).
\end{equation}
The condition that infinity be a regular singular point excludes a polynomial part. Since $Q_w=w^{-4}Q(1/w)$ for $w=1/z$, we obtain
\begin{equation}\label{eq:fuchs-relations}
 \sum_{i=1}^{r_0}c_i=0,\quad \sum_{i=1}^{r_0}(\delta_i+a_ic_i)=\delta_n.
\end{equation}
These two relations are independent because the finite marked points are distinct. The coefficient of $z^{-3}$ corresponds to an allowed $w^{-1}$ term at infinity. Thus the vector space of differences is
\begin{equation}\label{eq:quadratic-space}
 V=\left\{\frac{P(z)}{\prod_{i=1}^{r_0}(z-a_i)}\,dz^2:\deg P\le n-4\right\}.
\end{equation}
When $\A_X$ is empty or $m=0$, finiteness and properness are immediate, so we may assume $m>0$ in the sequel.

Fix a basis $e_1,\ldots,e_m$ of $V$ and equip $V$ with the Euclidean norm induced by this basis. On $\overline X$, we take the conjugate origin and basis. Convergence in $V$ implies uniform convergence of all derivatives of the local coefficients on compact subsets of $X\setminus D$. Consequently, after passing to a subsequence, any sequence $(Q_j)$ escaping every compact subset of $\A_X$ can be written as
\begin{equation}\label{eq:normalized}
 Q_j=Q_0+T_jq_j,\quad T_j=\norm{Q_j-Q_0}\longrightarrow\infty,\quad
 \norm{q_j}=1,\quad q_j\longrightarrow q\ne0.
\end{equation}
The limiting quadratic differential $q$ is nonzero, has finitely many zeros, and has at most simple poles at the marked points.

 
\subsection{The scalar equation and its monodromy}\label{subsec:scalar}

Fix a theta characteristic $L$, namely a holomorphic line bundle equipped with an isomorphism $L^{\otimes2}\simeq K_X$. Such a bundle exists on every compact Riemann surface \cite{For81}; on $\PP^1$ it is $\mathcal O(-1)$. Local solutions of the scalar equation are expressed as coefficients of sections of $L^{-1}$. Writing $t=dz/dw$ and choosing the square root according to the transition of $L$, we obtain
\begin{equation}\label{eq:solution-transition}
 y_w=t^{-1/2}y_z,\quad
 (y_w)_w=t^{1/2}(y_z)_z-\frac12t^{-3/2}t_wy_z.
\end{equation}
Substitution into $y_{zz}+Q_zy=0$ gives the equation with coefficient determined by \eqref{eq:projective-change}. The column vector consisting of a solution and its first derivative transforms according to
\begin{equation}\label{eq:jet-transition}
 G=\begin{pmatrix}t^{-1/2}&0\\-t_w/(2t^{3/2})&t^{1/2}\end{pmatrix},\quad \det G=1.
\end{equation}
These are the transition matrices of the first jet bundle $J^1(L^{-1})$, whose determinant is $K_X\otimes L^{-2}\simeq\mathcal O_X$. The choices of signs satisfy the cocycle condition because they come from $L$.

It follows that the equations define a rank-two local system with trivial determinant on $X^\circ=X\setminus D$. Fix a basepoint $x_0\in X^\circ$, a local frame there, and a basis of solutions normalized to have Wronskian one. The resulting representation is
\[
 \rho_Q:\Gamma\longrightarrow\SL_2(\C),
\]
where $\Gamma=\pi_1(X^\circ,x_0)$. 
Its projectivization is the monodromy of the developing map. When we choose a different normalized basis of solutions, we obtain a representation conjugate to $\rho_Q$; when we choose a different theta characteristic, we obtain a representation that differs from $\rho_Q$ by multiplication by a central sign character. In either case, the underlying projective connection is unchanged.

Choose handle loops $A_j,B_j$ and peripheral loops $\gamma_i$ so that the fundamental group
\begin{equation}\label{eq:surface-group}
 \Gamma=\left\langle A_1,B_1,\ldots,A_g,B_g,\gamma_1,\ldots,\gamma_n:
 \prod_{j=1}^g[A_j,B_j]\prod_{i=1}^n\gamma_i=1\right\rangle.
\end{equation}
For $n\ge1$, we eliminate one peripheral loop and denote the resulting free generators by $\sigma_1,\ldots,\sigma_r$, where $r=2g+n-1$. For $n=0$, we use the $r=2g$ handle generators with their relation. We fix this marking throughout.

\begin{definition}[Unitarizable projective monodromy]\label{def:unitarizable-projective-monodromy}
Let $\rho_Q^{\mathrm{proj}}:\Gamma\to\PSL_2(\C)$ be the projectivization of $\rho_Q$. The equation has \emph{unitarizable projective monodromy} if there is $M\in\PSL_2(\C)$ such that
\[
 M\rho_Q^{\mathrm{proj}}(\Gamma)M^{-1}\subset\PSU(2).
\]
\end{definition}

The scalar lift constructed above shows that this condition is equivalent to $\rho_Q$ being conjugate to a representation taking values in $\SU(2)$. Indeed, we can lift $M$ to $\SL_2(\C)$, and the inverse image of $\PSU(2)$ in $\SL_2(\C)$ is $\SU(2)$. This equivalence concerns the representation itself, whereas its character determines only its semisimplification.


\subsection{Characters and trace functions}\label{subsec:traces}

Let
\[
 \X(\Gamma)=\Hom(\Gamma,\SL_2(\C))\mathbin{/\mkern-6mu/}\SL_2(\C)
\]
be the affine character variety, and let $\pi_{\mathrm{ch}}$ be the quotient map. For every $\gamma\in\Gamma$, the function $[\rho]\mapsto\operatorname{tr}(\rho(\gamma))$ is regular on the character variety \cite{CS83,Gol09}. For a nonempty subset $I=\{i_1<\cdots<i_k\}\subset\{1,\ldots,r\}$, put
\begin{equation}\label{eq:finite-traces}
 t_I(\rho)=\tr\bigl(\rho(\sigma_{i_1})\cdots\rho(\sigma_{i_k})\bigr),\quad N=2^r-1.
\end{equation}
For every word in the generators, the corresponding trace function can be expressed as a polynomial with real coefficients in these $N$ functions \cite[Section~5.7]{Gol09}. The same calculation applies when the generators satisfy the standard relation for the fundamental group of a closed surface. To see this, let $C_i\in\SL_2(\C)$, and write $t_i=\tr C_i$ and $t_{ij}=\tr(C_iC_j)$. Then
\begin{equation}\label{eq:trace-identities}
 \begin{split}
 C_i^{-1}&=t_iI-C_i,\\
 C_iC_j+C_jC_i&=t_iC_j+t_jC_i+(t_{ij}-t_it_j)I\quad(i\ne j).
 \end{split}
\end{equation}
We first eliminate inverse letters using the first identity. We then argue by induction on word length and, for words of a fixed length, on the number of inversions in the sequence of indices. The second identity allows us to interchange adjacent factors whose indices are inverted, introducing only terms of smaller length. Once the indices are ordered, we use $C_i^2=t_iC_i-I$ to eliminate repetitions. Taking traces gives the desired polynomial expression. In particular, if the finite collection of traces is bounded, then the trace of every fixed word is bounded.

\begin{definition}[Holonomy and trace maps]\label{def:holonomy}
The \emph{holonomy map} assigns to each $Q\in\A_X$ the character of $\rho_Q$:
\[
 \Hol_X:\A_X\longrightarrow\X(\Gamma).
\]
Let $\tau:\X(\Gamma)\to\C^N$ be the map whose components are the trace functions in \eqref{eq:finite-traces}. The \emph{trace map} and its components are then given by
\begin{equation}\label{eq:trace-map}
 F_X=\tau\circ\Hol_X,\quad F_{X,I}=t_I\circ\Hol_X.
\end{equation}
\end{definition}

Both maps are holomorphic. To prove this for all accessory parameters, we continue normalized solutions along a finite chain of coordinate disks contained in $X\setminus D$ and covering each chosen loop. The coefficient matrix depends holomorphically on the parameters, while the transition matrices in \eqref{eq:jet-transition} are independent of them. Since the Picard series converges locally uniformly, each transfer matrix depends holomorphically on the parameters. Finite products of these matrices give monodromy matrices whose traces are entire functions on $\A_X\simeq\C^m$.

Equality of the finite traces implies equality of the traces of all words, and hence of the eigenvalues and hyperbolic translation lengths of the corresponding matrices. Equal characters also have equal finite traces. We shall use the following compact sets:
\begin{equation}\label{eq:compact-traces}
 \K=\pi_{\mathrm{ch}}\bigl(\Hom(\Gamma,\SU(2))\bigr),\quad K^{\mathrm{tr}}=\tau(\K).
\end{equation}
The chosen generators identify $\Hom(\Gamma,\SU(2))$ with a closed subset of $\SU(2)^r$, which is compact. Consequently, $\K$ and $K^{\mathrm{tr}}$ are compact, and every coordinate of $K^{\mathrm{tr}}$ is real.

We can obtain the corresponding objects on $\overline X$ by conjugating $Q_0$, the basis of $V$, the theta characteristic, and the initial local frame. We use the same paths on the underlying topological surface to mark the fundamental group of $\overline X$. Complex conjugation of analytic continuation then gives
\begin{equation}\label{eq:conjugate-traces}
 F_{\overline X}(\overline Q)=\overline{F_X(Q)}.
\end{equation}
The corresponding identity for holonomy characters follows in the same way. When $r=0$, the trace target consists of a single point; this occurs only in the elementary zero-dimensional cases considered above.


\section{Degeneration of holonomy}\label{sec:trees}

In this section, we fix a compact connected Riemann surface $X$, a reduced marked divisor $D$, prescribed coefficients $\boldsymbol\delta$, and a nonempty accessory space $\A_X=Q_0+V$, where $V=H^0(X,K_X^{\otimes2}(D))$. Put $X^\circ=X\setminus D$ and $\Gamma=\pi_1(X^\circ,x_0)$ for a fixed basepoint $x_0$. We use the theta characteristic, group marking, and coefficient norm fixed in Section~\ref{sec:holonomy}. We first treat the case $m=\dim V>0$. Throughout this section, constants may depend on these choices and on the specified compact coordinate sets and paths.

 
\subsection{Local estimates and coordinate changes}\label{sec:local}

We realize hyperbolic three-space $\HH^3$ as the space of positive definite Hermitian $2\times2$ matrices with determinant one \cite{BH99}. Each $C\in\SL_2(\C)$ acts by $H\mapsto CHC^*$, where $C^*$ denotes the conjugate transpose of $C$. We normalize the metric so that
\begin{equation}\label{eq:hyperbolic-norm}
 d_{\HH^3}(I,CC^*)=2\log\norm C_{\mathrm{op}}.
\end{equation}
Here $\norm C_{\mathrm{op}}$ denotes the operator norm with respect to the standard Hermitian form on $\C^2$. We denote the translation length of this isometry by $L(C)$. If $\lambda,\lambda^{-1}$ are the eigenvalues of $C$, then
\[
 L(C)=2\log\max\{|\lambda|,|\lambda|^{-1}\}.
\]

\begin{lemma}\label{lem:growth}
For each fixed word $\eta\in\Gamma$, there is $C_\eta>0$ such that
\begin{equation}\label{eq:word-growth}
 L(\rho_Q(\eta))\le C_\eta\sqrt{1+\norm{Q-Q_0}}\quad(Q\in\A_X).
\end{equation}
Moreover, for each $Q$, there is a hyperbolic basepoint at which the displacements of all chosen generators satisfy bounds of the same form simultaneously.
\end{lemma}

\begin{proof}
Put $\mu=\sqrt{1+\norm{Q-Q_0}}$. Along fixed representative paths covered by finitely many compact coordinate sets, $|Q_z|\le C\mu^2$. For the row solution matrix $Y_z$ with rows $(y_i,y_i')$, set $D_\mu=\operatorname{diag}(\mu^{1/2},\mu^{-1/2})$. Then
\[
 (Y_zD_\mu)'=(Y_zD_\mu)\begin{pmatrix}0&-Q_z/\mu\\ \mu&0\end{pmatrix}.
\]
Along these paths, the norm of this coefficient is at most $C\mu$ times the path speed. At a coordinate change, the balanced row matrix changes by $D_\mu^{-1}G^{\mathsf t}D_\mu$, with $G$ as in \eqref{eq:jet-transition}. Its diagonal entries are fixed and bounded away from zero; its off-diagonal entry is a fixed bounded function divided by $\mu$. Both this matrix and its inverse are bounded on the chosen overlaps. By Gronwall's inequality and the finiteness of the coordinate transitions, the norms of the transfer matrices and their inverses are bounded by $\exp(C_\eta\mu)$.

The transfer around a based loop is conjugate to its monodromy. Since its spectral radius is at most its norm, we obtain \eqref{eq:word-growth}. If $A$ is the balanced frame at the common initial point and $P_\eta$ is the loop transfer, then $\rho_Q(\eta)A=AP_\eta$. The displacement of $AA^*$ is $2\log\norm{P_\eta}_{\mathrm{op}}$. Using the same initial frame for all generators proves the last assertion.
\end{proof}

Consider a sequence $(Q_j)$ in $\A_X$ of the form
\begin{equation}\label{eq:general-sequence}
 Q_j=Q_0+T_jq_j,\quad T_j\longrightarrow+\infty,\quad q_j\longrightarrow q\ne0,
 \quad \lambda_j=\sqrt{T_j}.
\end{equation}
Here $T_j>0$ and $q_j\in V$. We do not require the positive scale $T_j$ to equal the accessory norm, so that two sequences can be compared using a common scale. On a regular coordinate disk choose $s_{j,z}=\sqrt{q_{j,z}}$, set $k_{j,z}=\lambda_js_{j,z}$, and define
\begin{equation}\label{eq:hyperbolic-frame}
 D_{j,z}=\begin{pmatrix}k_{j,z}^{1/2}&0\\0&k_{j,z}^{-1/2}\end{pmatrix},\quad
 A_{j,z}=Y_{j,z}D_{j,z},\quad H_{j,z}=A_{j,z}A_{j,z}^*.
\end{equation}
Here $Y_{j,z}$ is a matrix of determinant one formed from an analytically continued basis of local solutions. Different choices of square roots introduce diagonal unitary factors on the right, leaving $H_{j,z}$ unchanged. By contrast, monodromy acts on $Y_{j,z}$ by left multiplication.

We next check compatibility under changes of coordinates. For $z=z(w)$, we set $t=dz/dw$ and choose the square roots consistently so that $k_{j,w}=t k_{j,z}$. By \eqref{eq:jet-transition},
\begin{equation}\label{eq:balanced-transition}
 A_{j,w}=A_{j,z}E_j,\quad
 E_j=D_{j,z}^{-1}G^{\mathsf t}D_{j,w}
 =\begin{pmatrix}1&-t_w/(2t^2k_{j,z})\\0&1\end{pmatrix}.
\end{equation}
On a compact regular overlap, we have $E_j=I+O(\lambda_j^{-1})$. Since congruence by $A_{j,z}^{-1}$ is an isometry of $\HH^3$, we obtain
\begin{equation}\label{eq:coordinate-distance}
 d_{\HH^3}(H_{j,z},H_{j,w})=2\log\norm{E_j}_{\mathrm{op}}
 =O(\lambda_j^{-1}).
\end{equation}
The same estimate holds for other choices of square roots, since these introduce only unitary factors. It follows that the rescaled matrices determine the same limiting point on overlapping charts.

\begin{proposition}\label{prop:local-limit}
Let $U$ be a sufficiently small simply connected coordinate disk whose closure is disjoint from $D$ and the zeros of $q$. For $z,w$ in a compact smaller regular flat rectangle,
\begin{equation}\label{eq:local-limit}
 \frac{d_{\HH^3}(H_j(z),H_j(w))}{2\lambda_j}
 \longrightarrow\left|\operatorname{Im}\int_w^z\sqrt q\right|.
\end{equation}
The convergence is uniform as both endpoints vary.
\end{proposition}

\begin{proof}
Choose square roots $s_j=\sqrt{q_j}$ that converge, together with all their derivatives, to a nonvanishing square root $s=\sqrt q$. After shrinking the disk, we may assume that the maps $\zeta_j=\int s_j\,dz$ are injective. For sufficiently large $j$, their images contain every compact rectangle lying strictly inside the image of the limiting coordinate.

By the Liouville substitution $y=s_j^{-1/2}v$, one has
\begin{equation}\label{eq:liouville}
 v_{\zeta\zeta}+(\lambda_j^2+W_j)v=0,\quad
 W_j=\frac{Q_0}{s_j^2}+\frac{(s_j^{-1/2})''}{s_j^{3/2}},
\end{equation}
where the last derivatives are taken in $z$. The functions $W_j$ are uniformly bounded on compact smaller rectangles. The Wronskian is unchanged under this substitution because $(s_j^{-1/2})^2\zeta_j'=1$. Let $B_j$ be the matrix whose $i$th row is $(\lambda_j^{1/2}v_{i,j},\lambda_j^{-1/2}v_{i,j,\zeta})$. Direct differentiation gives
\begin{equation}\label{eq:shear}
 A_j=B_jS_j,\quad S_j=\begin{pmatrix}1&-s_j'/(2\lambda_js_j^2)\\0&1\end{pmatrix},
\end{equation}
and hence
\begin{equation}\label{eq:shear-distance}
 d_{\HH^3}(A_jA_j^*,B_jB_j^*)=d_{\HH^3}(S_jS_j^*,I)=O(\lambda_j^{-1}).
\end{equation}

We parametrize the straight line segment joining the endpoints by $\zeta_j(t)=\zeta_j(0)+\ell_jt$, $0\le t\le1$. Along this segment, the matrix equation takes the form
\begin{equation}\label{eq:right-ode}
 \begin{gathered}
 B_j'=B_j(M_j+R_j(t)),\\
 M_j=\lambda_j\ell_j\begin{pmatrix}0&-1\\1&0\end{pmatrix},\quad
 R_j(t)=\begin{pmatrix}0&-\ell_jW_j/\lambda_j\\0&0\end{pmatrix}.
 \end{gathered}
\end{equation}
Writing $B_j(t)=B_j(0)E_j(t)$, we obtain by variation of constants
\begin{equation}\label{eq:duhamel}
 E_j(t)=e^{tM_j}+\int_0^t E_j(u)R_j(u)e^{(t-u)M_j}\,du.
\end{equation}
The matrix $M_j$ is normal and $\norm{e^{tM_j}}=e^{h_jt}$ for $t\ge0$, where $h_j=\lambda_j|\operatorname{Im}\ell_j|$. Since $\norm{R_j}\le C/\lambda_j$, we first apply Gronwall's inequality to obtain
\[
e^{-h_jt}\norm{E_j(t)}\le e^{Ct/\lambda_j}.
\]
Substituting this bound into \eqref{eq:duhamel} yields
\begin{equation}\label{eq:relative-bound}
 \norm{E_j(1)-e^{M_j}}\le e^{h_j}(e^{C/\lambda_j}-1).
\end{equation}
With $\eta_j=e^{C/\lambda_j}-1$, the triangle inequality and the reverse triangle inequality give
\[
 (1-\eta_j)e^{h_j}\le\norm{E_j(1)}\le(1+\eta_j)e^{h_j}
\]
for sufficiently large $j$. Consequently,
\begin{equation}\label{eq:singular-value}
 \log\norm{E_j(1)}=\lambda_j|\operatorname{Im}\ell_j|+O(\lambda_j^{-1}).
\end{equation}
Since congruence by $B_j(0)^{-1}$ is a hyperbolic isometry, the common factor $B_j(0)$ does not affect the distance. By \eqref{eq:hyperbolic-norm} and \eqref{eq:shear-distance}, we obtain
\[
 d_{\HH^3}(H_j(z),H_j(w))=2\lambda_j|\operatorname{Im}\ell_j|+O(\lambda_j^{-1}).
\]
As the endpoints vary in the chosen compact rectangle, $\ell_j$ converges uniformly to $\int_w^z\sqrt q$. This proves \eqref{eq:local-limit}.
\end{proof}

We next use the integrability of $|q|^{1/2}$ to approach the zeros and poles of the limiting differential in the limit tree.

 
\subsection{The limiting tree}\label{subsec:tree-limit}

\begin{definition}[Real trees and translation lengths]\label{def:tree-length}
An $\R$-tree is a metric space in which any two points are joined by a unique arc and that arc is isometric to a real interval. For an isometry $h$ of a metric space $(M,d)$, its translation length is $\inf_{x\in M}d(x,hx)$. We write $\ell_T(h)$ for this length in a tree $T$. For a marked group action, we call $\eta\mapsto\ell_T(\eta)$ the \emph{marked length function}.
\end{definition}

Choose a lift $z_0$ of a point where $q$ is regular and nonzero. Fix a nonprincipal ultrafilter $\mathfrak u$ on $\mathbb N$, and form
\begin{equation}\label{eq:cone}
 T=\lim_{\mathfrak u}\bigl(\HH^3,(2\lambda_j)^{-1}d_{\HH^3},H_j(z_0)\bigr).
\end{equation}
The metric ultralimit is complete \cite[Lemma~I.5.53]{BH99} and geodesic. Since the hyperbolicity constant tends to zero under rescaling, the ultralimit is an $\R$-tree \cite{BH99,KL95}.

\begin{proposition}\label{prop:tree-map}
The monodromy representations induce an isometric action of $\Gamma$ on $T$ and a nonconstant equivariant map
\[
 f:\widetilde{X^\circ}\longrightarrow T.
\]
On a regular flat rectangle,
\begin{equation}\label{eq:height-model}
 f=\iota\circ u,\quad u=\operatorname{Im}\int\sqrt q,
\end{equation}
where $\iota$ is an isometric embedding of a real interval. The map extends continuously across the lifts of the zeros of $q$. Each lifted puncture has a finite limiting endpoint fixed by its peripheral element.
\end{proposition}

\begin{proof}
Since $q_j\to q\ne0$ in $V$, the quantity $\sqrt{1+\norm{Q_j-Q_0}}$ is comparable with $\lambda_j$. At each fixed regular point, the diagonal matrices in \eqref{eq:hyperbolic-frame} are uniformly comparable to those used for the coarse balancing in Lemma~\ref{lem:growth}. Finite representative paths therefore give $O(\lambda_j)$ based displacements for every generator, and for every fixed word. It follows that the monodromy actions induce an action on the pointed ultralimit \eqref{eq:cone}.

For a regular lifted point $z$, continue the fundamental solutions along a fixed path and take the ultralimit of $H_j(z)$. The same coarse estimate bounds its rescaled distance from the basepoint. The coordinate discrepancy \eqref{eq:coordinate-distance} vanishes after rescaling, so the point is independent of the coordinate chart. The theta cocycle and analytic continuation ensure compatibility along chart chains and equivariance. By Proposition~\ref{prop:local-limit}, two points of a regular rectangle have limiting distance $|u(z)-u(w)|$. Points at equal heights have the same image, while distinct heights determine an isometric interval, as illustrated in Figure~\ref{fig:local-height}. This proves \eqref{eq:height-model} and nonconstancy.

Subdivision of any regular piecewise smooth path gives
\begin{equation}\label{eq:flat-path}
 d_T(f(x),f(y))\le\int_{x}^{y}|q|^{1/2}.
\end{equation}
Near a point where $q=z^kh(z)\,dz^2$, $h(0)\ne0$, with $k\ge-1$, radial tails and circular arcs of bounded angle have flat length $O(r^{1+k/2})$. At an ordinary zero, completeness of $T$ then gives a unique continuous extension. At a marked simple pole they give a radial endpoint and an $O(\sqrt r)$ bound in a fundamental angular sector. A full peripheral turn has length tending to zero, so equivariance makes the peripheral element fix that endpoint. Since its powers also fix the endpoint, the same bound holds in every lifted angular sector. At a mark where the simple pole cancels, the corresponding holomorphic estimate is stronger.
\end{proof}

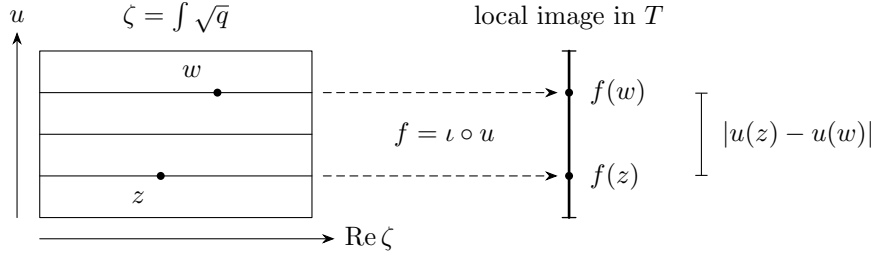
\begin{figure}[htbp]
\centering
\begin{tikzpicture}[x=1cm,y=1cm,font=\small,
  line cap=round,line join=round,>=Stealth]
  \draw[line width=.4pt] (0,0) rectangle (3.6,2.2);
  \foreach \h in {.55,1.1,1.65}
    \draw[line width=.4pt] (0,\h)--(3.6,\h);
  \draw[->,line width=.4pt] (0,-.28)--(3.9,-.28)
    node[right] {$\operatorname{Re}\zeta$};
  \draw[->,line width=.4pt] (-.3,0)--(-.3,2.45)
    node[above] {$u$};
  \node at (1.8,2.65) {$\zeta=\int\sqrt q$};
  \fill (1.6,.55) circle (1.5pt) node[below left=2pt] {$z$};
  \fill (2.35,1.65) circle (1.5pt) node[above left=2pt] {$w$};
  \draw[->,densely dashed,line width=.4pt] (3.75,.55)--(6.85,.55);
  \draw[->,densely dashed,line width=.4pt] (3.75,1.65)--(6.85,1.65);
  \node at (5.35,1.1) {$f=\iota\circ u$};
  \draw[line width=1pt] (7,0)--(7,2.2);
  \draw[line width=.5pt] (6.91,0)--(7.09,0);
  \draw[line width=.5pt] (6.91,2.2)--(7.09,2.2);
  \fill (7,.55) circle (1.5pt) node[right=4pt] {$f(z)$};
  \fill (7,1.65) circle (1.5pt) node[right=4pt] {$f(w)$};
  \node at (7,2.65) {local image in $T$};
  \draw[|-|,line width=.4pt] (8.75,.55)--(8.75,1.65)
    node[midway,right=4pt] {$|u(z)-u(w)|$};
\end{tikzpicture}
\caption{The local height model (\ref*{eq:height-model}): horizontal trajectories in a regular flat rectangle collapse to points.}
\label{fig:local-height}
\end{figure}

In particular, every invariant scalar distance formed from these maps and an invariant closed target set is bounded near the punctures and ordinary zeros. These bounds apply to the completed limit maps.

 
\subsection{Properness and growth}\label{subsec:properness}

To identify the exact scale of the limiting length function, we first relate translation length to two displacements from a fixed basepoint.

\begin{lemma}\label{lem:displacement}
For an isometry $h$ of an $\R$-tree and a point $x$,
\begin{equation}\label{eq:tree-displacement}
 \ell_T(h)=\max\{0,d_T(x,h^2x)-d_T(x,hx)\}.
\end{equation}
For an orientation-preserving isometry of $\HH^3$, with $d_\nu=d(P,h^\nu P)$,
\begin{equation}\label{eq:hyperbolic-displacement}
 \left|\max\{0,d_2-d_1\}-L(h)\right|\le4.
\end{equation}
\end{lemma}

\begin{proof}
For a hyperbolic tree isometry the two displacements are $2r+2\ell_T(h)$ and $2r+\ell_T(h)$, where $r$ is distance to its axis. For an elliptic isometry, displacement is twice the distance to its fixed subtree and $\Fix(h)\subset\Fix(h^2)$. The standard formulas for tree isometries \cite[Section~1.3]{CM87} therefore give \eqref{eq:tree-displacement}.

In $\HH^3$, suppose first that $h$ is loxodromic or nonidentity elliptic, with axis distance $r$ and rotation angle $\theta$. Set $L=L(h)$ and $A_\nu=\sinh(d_\nu/2)$. Cylindrical coordinates give
\begin{equation}\label{eq:cylindrical}
 A_\nu^2=\sinh^2(\nu L/2)\cosh^2r+\sin^2(\nu\theta/2)\sinh^2r.
\end{equation}
Thus $A_2\le2\cosh(L/2)A_1$, and $\arsinh(kx)-\arsinh x\le\log k$ for $k\ge1$ gives
\begin{equation}\label{eq:displacement-upper}
 d_2-d_1\le2\log(2\cosh(L/2))\le L+2\log2.
\end{equation}
For $L\ge1$ we also have
\[
 A_1\le\cosh(L/2)\cosh r,\quad
 A_2\ge2\sinh(L/2)A_1,\quad A_1\ge\sinh(1/2)>1/2.
\]
Using $\arsinh(kx)\ge\log(2kx)$ and $\arsinh x\le\log(2x)+\log((1+\sqrt5)/2)$ for $x\ge1/2$, we obtain
\[
 d_2-d_1\ge L+2\log(1-e^{-L})-2\log\frac{1+\sqrt5}{2}>L-4.
\]
When $L<1$, the lower bound follows from the nonnegativity of the maximum. In the parabolic case $L=0$ and $A_2=2A_1$, which gives the same upper bound; the identity is immediate. This proves \eqref{eq:hyperbolic-displacement}.
\end{proof}

Applying these estimates at the basepoint of \eqref{eq:cone}, we see that the error disappears after rescaling and that both based displacements converge to their ultralimits. For every $\eta\in\Gamma$,
\begin{equation}\label{eq:length-limit}
 \ell_T(\eta)=\lim_{\mathfrak u}\frac{L(\rho_{Q_j}(\eta))}{2\lambda_j}.
\end{equation}

\begin{lemma}\label{lem:nonzero-length}
The action on $T$ has a nonzero marked length function.
\end{lemma}

\begin{proof}
If all lengths vanish, all group elements are elliptic. The fixed subtrees of two generators intersect, since otherwise their product has positive length. By the finite Helly property for subtrees, the generators have a common fixed point $a$; see \cite[Sections~1.2--1.5]{CM87}.

The function $d_T(f,a)^2$ descends to $X^\circ$. On a regular rectangle it is a convex function of the real harmonic coordinate in \eqref{eq:height-model}, hence subharmonic. A continuous convex function of finitely many harmonic functions is subharmonic: this follows by the chain rule for a smooth convex function and by local convolution in general. The endpoint bounds allow us to extend the function subharmonically across every omitted point. By the maximum principle on the compact connected curve, the resulting function is constant \cite[Chapters~2--3]{Ran95}. Distance to a fixed point in a tree cannot be constant on a nondegenerate interval, contradicting the local form of $f$.
\end{proof}

\begin{proposition}\label{prop:properness}
For arbitrary fixed complex coefficients $\boldsymbol\delta$, the maps $F_X$ and $\Hol_X$ are proper.
\end{proposition}

\begin{proof}
An escaping sequence with bounded finite traces has bounded traces for every fixed word by \eqref{eq:trace-identities}. The eigenvalues of its monodromy matrices, and hence their translation lengths, are bounded word by word. Normalize the sequence as in \eqref{eq:normalized}. Equation~\eqref{eq:length-limit} makes every length in the limiting tree zero, contrary to Lemma~\ref{lem:nonzero-length}. Thus the inverse image under $F_X$ of a compact set is bounded. It is closed by continuity, hence compact in $\A_X\simeq\C^m$.

For compact $C\subset\X(\Gamma)$, the closed set $\Hol_X^{-1}(C)$ is contained in the compact set $F_X^{-1}(\tau(C))$. Hence the holonomy map is proper as well.
\end{proof}

\begin{corollary}\label{cor:finite-holonomy}
Every complex character has finitely many realizations in $\A_X$. With its reduced analytic structure, the image $\mathcal Y_X=\Hol_X(\A_X)$ is closed, irreducible, and of dimension $m$, and $\Hol_X:\A_X\to\mathcal Y_X$ is finite holomorphic. The same conclusions hold for the trace map onto its image.
\end{corollary}

\begin{proof}
Each fiber is a compact analytic subset of an affine complex space and is therefore finite. Indeed, on any hypothetical positive-dimensional component the strictly plurisubharmonic function $\norm z^2$ attains a maximum, contradicting the maximum principle on a local nonconstant analytic curve through that point \cite{Chi89}. A proper holomorphic map with finite fibers is finite, and its image is analytic by Remmert's theorem \cite[Chapter~I, Section~2]{Fre16}. The image is irreducible because the source is irreducible, and it has dimension $m$ by the dimension theorem for finite maps.
\end{proof}

A prescribed conjugacy class of scalar representations has finitely many realizations because they lie in a single character fiber. The same holds for a prescribed projective representation: any existing scalar lifts differ by elements of the finite group $\Hom(\Gamma,\{\pm1\})$. A prescribed affine projective character also has a finite fiber. Indeed, the regular functions $t_I^2$ are independent of the choice of scalar lift. Fixing the projective character fixes these squares, leaving at most $2^N$ possible finite-trace vectors, each with a finite fiber. Since coordinatewise squaring is finite and proper, $(F_{X,I}^2)_I$ is also finite holomorphic onto its reduced analytic image.

The cardinalities of the fibers of a fixed map are locally bounded: the coherent algebra $(\Hol_X)_*\mathcal O_{\A_X}$ has finitely many generators locally, which bound the dimension of its fiber algebra and hence the number of distinct points \cite[Chapter~I, Section~2]{Fre16}. Taking a finite cover then gives a bound over any compact subset of the target. The image is Stein as a closed analytic subspace of the affine character variety. If $\nu:\mathcal Y_X^\nu\to\mathcal Y_X$ is its normalization, normality of $\A_X$ gives a finite factorization through $\mathcal Y_X^\nu$. Writing $d_h$ for the generic degree of this finite map, finite mapping theory gives
\[
 \#\Hol_X^{-1}(\chi)\le d_h\,\#\nu^{-1}(\chi).
\]
The normalization accounts for the branches of the image at a singular point; the map is an unramified finite covering off its critical values on the regular image \cite{Chi89,Fre16}.

Define the logarithmic size of the finite traces by
\begin{equation}\label{eq:log-trace-size}
 \Lambda_X(Q)=\log\left(2+\max_I|F_{X,I}(Q)|\right).
\end{equation}
If the trace list is empty, we define the maximum to be zero. We extend the growth estimate for closed curves in \cite[Theorem~5.2]{Dum17} to the present setting of regular-singular equations.

\begin{proposition}\label{prop:trace-growth}
There are constants $c,C,R>0$ such that
\begin{equation}\label{eq:trace-growth}
 c\sqrt{\norm{Q-Q_0}}\le\Lambda_X(Q)\le C\sqrt{1+\norm{Q-Q_0}}
\end{equation}
for $\norm{Q-Q_0}\ge R$. In particular, for some $C_0>0$,
\begin{equation}\label{eq:accessory-trace-bound}
 \norm{Q-Q_0}\le C_0\bigl(1+\Lambda_X(Q)^2\bigr)\quad(Q\in\A_X).
\end{equation}
\end{proposition}

\begin{proof}
For $C\in\SL_2(\C)$ choose an eigenvalue $\lambda$ of modulus at least one. Then
\[
 |\tr C|\le2e^{L(C)/2},\quad |\lambda|\le|\tr C|+1.
\]
Applying Lemma~\ref{lem:growth} to the finitely many coordinate words gives the upper bound everywhere. If the lower bound fails, choose $Q_j$ with $T_j=\norm{Q_j-Q_0}\to\infty$ and $\Lambda_X(Q_j)/\sqrt{T_j}\to0$, and extract a nonzero normalized limit $q$. Since each word trace is a fixed polynomial in the finite traces,
\[
 \log\bigl(2+|\tr\rho_{Q_j}(\eta)|\bigr)\le C_\eta\bigl(1+\Lambda_X(Q_j)\bigr).
\]
By the eigenvalue bound, every normalized word length tends to zero. This contradicts \eqref{eq:length-limit} and Lemma~\ref{lem:nonzero-length}. We obtain \eqref{eq:accessory-trace-bound} by squaring the lower estimate and increasing the constant on a compact ball.
\end{proof}

Consequently, a continuous family of connections over a connected parameter space is constant if its characters lie in a fixed finite set. An entire holomorphic family $\C^k\to\A_X$ with relatively compact character image is also constant: properness bounds all its coefficient functions, and Liouville's theorem applies.

 
\subsection{Smooth marked families}\label{subsec:relative}

We now establish uniform estimates for the counting results in families. Let $\pi:\mathcal C\to B$ be a smooth proper holomorphic family of compact connected curves, with pairwise disjoint holomorphic marked sections. We work in a local marked parameter chart $B$ with a holomorphic theta characteristic and a compatible scalar lift. Suppose we have a holomorphic affine family of accessory spaces
\begin{equation}\label{eq:relative-accessory}
 \A_b=Q_0(b)+V_b,\quad V_b=H^0(X_b,K_{X_b}^{\otimes2}(D_b)),\quad X_b=\pi^{-1}(b),
\end{equation}
with fixed coefficients $\boldsymbol\delta$, constant rank $m$, a holomorphic origin $Q_0$, and a holomorphic frame $e_1(b),\ldots,e_m(b)$. We use the coefficient norm defined by this frame. Through the marking, we identify the fundamental groups with a fixed $\Gamma$ and fix generators $\sigma_1,\ldots,\sigma_r$. Put
\[
 \mathcal W=\{\sigma_i:1\le i\le r\}\cup\{\sigma_i\sigma_k:1\le i<k\le r\}.
\]

\begin{proposition}\label{prop:relative-properness}
In the setting \eqref{eq:relative-accessory}, the relative maps
\[
 (b,Q)\longmapsto(b,\Hol_{X_b}(Q)),\quad
 (b,Q)\longmapsto(b,F_{X_b}(Q))
\]
are holomorphic and proper. For each compact $K\Subset B$ and $m>0$, there are $c_K,C_K,R_K>0$ such that, for $b\in K$ and $T=\norm{Q-Q_0(b)}\ge R_K$,
\begin{equation}\label{eq:relative-growth}
 c_K\sqrt T\le\log\left(2+\max_{w\in\mathcal W}|\tr\rho_Q(w)|\right)
 \le C_K\sqrt{1+T}.
\end{equation}
\end{proposition}

\begin{proof}
We begin with the geometric construction needed for uniformity. A proper smooth submersion with disjoint sections admits a local smooth trivialization that preserves the sections. To construct such a trivialization, we choose horizontal lifts of coordinate fields on the base and add vertical corrections so that the lifts agree with the derivatives of the sections. Since the section images have disjoint neighborhoods, we can choose these corrections independently. Over a smaller coordinate ball, the time-one flow of the lift of the radial field from its center stays over a compact base set; properness gives existence of the flow and its inverse. The resulting diffeomorphisms preserve each section and depend smoothly on the base parameter. We can therefore transport fixed topological paths through finitely many relatively compact relative charts, with uniformly bounded $C^1$ coordinate parametrizations.

Suppose $b_j\to b_0$ and $Q_j=Q_0(b_j)+T_jq_j$, where $T_j=\norm{Q_j-Q_0(b_j)}\to\infty$. After passing to a subsequence, the normalized coefficient vectors converge. In the holomorphic frame of $V_b$, their limit defines $q\ne0$ on $X_{b_0}$. On compact relative coordinate disks, the coefficients of $q_j$ and all their fiber derivatives converge to those of $q$, by holomorphicity and Cauchy's estimates. At a marked section $z=0$, the functions $zq_{j,z}$ are holomorphic and converge in the same manner, so $q$ has at most simple poles.

Along the transported compact paths, the coefficients of $Q_0(b_j)$, the transition functions, and their derivatives are uniformly bounded; the coordinate derivatives are also bounded away from zero. With $\lambda_j=\sqrt{T_j}$, the coarse balanced transition has the form
\[
 \begin{pmatrix}t^{-1/2}&-t_w/(2\lambda_j t^{3/2})\\0&t^{1/2}\end{pmatrix}.
\]
Both this matrix and its inverse have uniformly bounded norms. Gronwall therefore gives uniform $O(\lambda_j)$ generator displacements. On any compact regular disk of the limit differential, the roots $\sqrt{q_j}$, their natural coordinates, and their inverse coordinate derivatives converge uniformly with the required derivatives. The functions $W_j$ in \eqref{eq:liouville} are uniformly bounded there. Consequently, the constants in \eqref{eq:shear-distance}--\eqref{eq:singular-value} can be chosen uniformly along this moving sequence. This proves the local height limit for transported points on the limiting curve, with the exact scale $1/(2\lambda_j)$.

We form a single pointed ultralimit and use the transported points to define a map from the punctured universal cover of $X_{b_0}$. The transition estimates ensure that these definitions agree on overlaps, while the marking gives equivariance. Its local image is the isometric height interval of $q$. Inequality~\eqref{eq:flat-path} now holds in this limit tree. The radial and angular estimates for the limit differential extend the map across its zeros and give peripheral endpoints, uniformly over all angular sectors after using the peripheral isometry. These estimates apply to the map in the limiting tree. The compact-surface maximum argument of Lemma~\ref{lem:nonzero-length}, on $X_{b_0}$, proves nonzero length. In fact, some word in the finite set $\mathcal W$ has positive length: if all these lengths vanished, the generators would be elliptic with pairwise intersecting fixed subtrees, and Helly's property would give a common fixed point. Finally, Lemma~\ref{lem:displacement} gives the exact wordwise length comparison in the moving family.

Monodromy is jointly holomorphic in $(b,Q)$. To verify this, we follow each marked loop through finitely many relative disks, use holomorphic point sections on overlaps to continue the initial-value solution, and apply the locally uniformly convergent Picard series to the coefficient and jet transition matrices. Local changes of representatives only conjugate monodromy. Hence traces and characters are holomorphic on the whole chart.

Suppose a sequence remains over a compact subset of the base and character space, but its coefficients are unbounded. After passing to a subsequence, we can apply the preceding construction. Each word in $\mathcal W$ has bounded trace, so \eqref{eq:length-limit} gives zero length on $\mathcal W$, a contradiction. The same reasoning applies directly to bounded finite traces. Hence the coefficients are bounded over compact target sets locally. Finite-dimensional compactness and continuity give properness of both relative maps. Rank zero is immediate from the affine origin.

To prove the lower estimate in \eqref{eq:relative-growth}, suppose that a sequence $b_j\in K$, $T_j\to\infty$, satisfies
\[
 \max_{w\in\mathcal W}\frac{L(\rho_{Q_j}(w))}{2\sqrt{T_j}}\longrightarrow0.
\]
Compactness of $K$ and the normalized coefficient sphere gives the limit construction above, contradicting its positive length on $\mathcal W$. We therefore obtain a uniform positive lower bound beyond a fixed threshold. If $w$ realizes this lower bound, then
\[
 |\tr\rho_Q(w)|\ge e^{c\sqrt T}-e^{-c\sqrt T}.
\]
After increasing the threshold if necessary, we obtain the lower logarithmic estimate. The upper estimate follows from the uniform compact-path Gronwall bounds and $|\tr C|\le2e^{L(C)/2}$.
\end{proof}

These conclusions hold on compact subsets of a smooth marked chart with the data in \eqref{eq:relative-accessory}. This is also a uniform estimate for $\Lambda_{X_b}$, since its finite trace list contains $\mathcal W$.


\section{Comparison of limiting quadratic differentials}\label{sec:comparison}

Let $X$ be a compact connected Riemann surface, let $D$ be its marked set, and put $X^\circ=X\setminus D$ and $\Gamma=\pi_1(X^\circ,x_0)$. We consider the action on a complete tree $T$ and the map $f:\widetilde{X^\circ}\to T$ constructed in Section~\ref{sec:trees} from a nonzero limiting differential $q$. The marked length function is denoted by $\ell_T$. It is nonzero by Lemma~\ref{lem:nonzero-length}. We seek to recover $q$ from the information shared by two such actions, using the exact normalization in \eqref{eq:length-limit}.

 
\subsection{Minimal subtrees and harmonic coordinates}\label{subsec:minimal}

\begin{definition}[Abelian length functions]\label{def:abelian-length}
A marked length function of an action of $\Gamma$ on an $\R$-tree is \emph{abelian} if it has the form $\eta\mapsto|c(\eta)|$ for a homomorphism $c:\Gamma\to\R$. Otherwise it is \emph{nonabelian}. We follow the terminology of Culler--Morgan for semisimple actions. In particular, actions preserving a line, including those with reflections, are semisimple \cite{CM87}.
\end{definition}

\begin{lemma}\label{lem:minimal-tree}
The action on $T$ has a unique minimal invariant subtree $T_{\min}$, and $f$ takes values in its completion $T_c$, realized as its closure in $T$. Two such actions with the same nonabelian marked length function have equivariantly isometric completed minimal subtrees.
\end{lemma}

\begin{proof}
The existence and uniqueness of $T_{\min}$ follow from the nonzero length function and \cite[Proposition~3.1]{CM87}. Its closure $T_c$ in the complete tree is closed, convex, and invariant. The function $d_T(f,T_c)^2$ descends to $X^\circ$. On a regular rectangle, it is a convex function of the harmonic height $u=\operatorname{Im}\int\sqrt q$, hence is subharmonic. The endpoint estimates in Proposition~\ref{prop:tree-map} give boundedness at the punctures and zeros. Removability and the maximum principle make it constant on $X$ \cite[Chapters~2--3]{Ran95}. A level set of positive distance from a convex subtree contains no nondegenerate interval: along a segment outside the subtree the distance has slope $1$ or $-1$, apart from a possible minimum. Since $f$ has a nondegenerate local interval as image, the constant is zero. Hence $f$ takes values in $T_c$.

A nonsemisimple action has an abelian length function by \cite[Corollary~2.4]{CM87}. The actions in the last assertion are therefore semisimple. On their minimal subtrees, \cite[Theorem~3.7]{CM87} gives an equivariant isometry. It extends uniquely to metric completions, which proves the assertion.
\end{proof}

On the sphere, the peripheral argument excludes abelian lengths altogether. Each peripheral loop has zero length by Proposition~\ref{prop:tree-map}. Since these loops generate $\pi_1(\PP^1\setminus D)$, a homomorphism $c$ with $\ell_T=|c|$ would vanish identically. This contradicts Lemma~\ref{lem:nonzero-length}. For a general compact Riemann surface, we instead treat abelian lengths through their periods.

We first compare maps whose completed target trees have been identified. The argument uses only their local interval descriptions and endpoint bounds.

\begin{proposition}\label{prop:constant-distance}
Let $f_1,f_2:\widetilde{X^\circ}\to T$ be nonconstant maps equivariant for the same action of $\Gamma$ on a complete $\R$-tree. Suppose that outside the lifts of a finite set they have the form $f_1=\iota\circ u$ and $f_2=\jmath\circ v$, where $u,v$ are real harmonic coordinates with nonzero differential and $\iota,\jmath$ are isometric interval embeddings. Suppose also that the descended function $d_T(f_1,f_2)^2$ is bounded near every puncture and exceptional point. Then this function is constant, and $(\partial u)^2=(\partial v)^2$ on each common regular disk.
\end{proposition}

\begin{proof}
On a common regular disk, the function
\begin{equation}\label{eq:convex-distance}
 (s,t)\longmapsto d_T(\iota(s),\jmath(t))^2
\end{equation}
is convex. Indeed, distance between points moving along geodesics in a tree is convex; the same fact holds in $\operatorname{CAT}(0)$ spaces \cite[Chapter~II.2]{BH99}. Squaring preserves convexity for a nonnegative convex function. Composition with the harmonic map $(u,v)$ to $\R^2$ is subharmonic. For a nonsmooth convex function, we obtain the same conclusion by convolution on smaller rectangles and passage to the distributional limit.

Equivariance makes $d_T(f_1,f_2)^2$ single-valued on $X^\circ$. Boundedness at the exceptional points permits subharmonic extension across them. The resulting function is constant on the compact curve; write $d_T(f_1,f_2)=c$. If $c=0$, the maps agree, and their local interval coordinates differ by an isometry of intervals. Hence $du=\pm dv$.

Suppose $c>0$. Near a common regular point, shrink the disk so that the two compact image intervals are disjoint. Their unique joining segment has endpoints with interval coordinates $s_*$ and $t_*$. If $b$ is its length, then
\begin{equation}\label{eq:bridge}
 d_T(\iota(s),\jmath(t))=|s-s_*|+b+|t-t_*|.
\end{equation}
The endpoints of the joining segment may lie in the interiors of the two intervals, as illustrated in Figure~\ref{fig:interval-bridge}. The level sets $u=s_*$ and $v=t_*$ are closed and have empty interior, because both differentials are nonzero. On a nonempty connected subdisk avoiding their union, the two signs in \eqref{eq:bridge} are constant. Since the distance is constant, $du=\pm dv$ there. By squaring the $(1,0)$ parts and applying the identity theorem for holomorphic quadratic differentials, we extend the equality to the common regular disk.
\end{proof}

\begin{figure}[htbp]
\centering
\begin{tikzpicture}[x=1cm,y=1cm,font=\small,
  line cap=round,line join=round]
  \draw[line width=.45pt] (0,-1.55)--(0,1.55);
  \draw[line width=.45pt] (5.5,-1.55)--(5.5,1.55);
  \foreach \x in {0,5.5}
    \foreach \y in {-1.55,1.55}
      \draw[line width=.45pt] (\x-.09,\y)--(\x+.09,\y);
  \draw[line width=1pt] (0,1.1)--(0,0)--(5.5,0)--(5.5,-1.1);
  \fill (0,1.1) circle (1.5pt) node[left=5pt] {$\iota(s)$};
  \fill (5.5,-1.1) circle (1.5pt) node[right=5pt] {$\jmath(t)$};
  \draw[fill=white,line width=.6pt] (0,0) circle (2pt);
  \draw[fill=white,line width=.6pt] (5.5,0) circle (2pt);
  \node[below left=4pt] at (0,0) {$\iota(s_*)$};
  \node[above right=4pt] at (5.5,0) {$\jmath(t_*)$};
  \node[right=5pt] at (0,.6) {$|s-s_*|$};
  \node[above=4pt] at (2.75,0) {$b$};
  \node[left=5pt] at (5.5,-.6) {$|t-t_*|$};
\end{tikzpicture}
\caption{The bridge decomposition (\ref*{eq:bridge}), with both gates in the interiors of the two image intervals. The heavier path has length $|s-s_*|+b+|t-t_*|$.}
\label{fig:interval-bridge}
\end{figure}
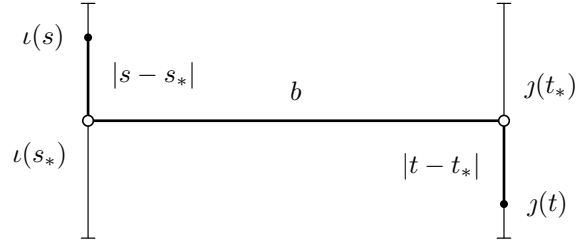

 
\subsection{Abelian lengths and harmonic periods}\label{subsec:abelian}

We next recover the differential without identifying the target trees. The following argument applies on every compact $X$; on the sphere its conclusion is incompatible with $q\ne0$ because there are no nonzero holomorphic one-forms. Dumas's Theorem~C gives the corresponding reconstruction for limits of unbranched projective structures on closed curves \cite{Dum17}. In our setting, the endpoint estimates allow the argument to extend across the prescribed regular singularities.

An \emph{end} of a tree is an equivalence class of geodesic rays which eventually have a common tail. For a ray $r:[0,\infty)\to T$ representing an end, its Busemann function is
\begin{equation}\label{eq:busemann}
 b_T(x)=\lim_{t\to\infty}\bigl(d_T(x,r(t))-t\bigr).
\end{equation}
A different choice of ray changes $b_T$ only by an additive constant. The function is convex on segments and $1$-Lipschitz \cite[Chapter~II.8]{BH99}.

\begin{proposition}\label{prop:abelian-reconstruction}
Suppose the nonzero marked length function of the limiting action is abelian. There is a smooth real harmonic one-form $\alpha$ on $X$, with $(1,0)$ part denoted by $\alpha^{1,0}$, whose period homomorphism $c:\Gamma\to\R$ satisfies $|c(\eta)|=\ell_T(\eta)$, such that
\begin{equation}\label{eq:abelian-square}
 q=-4(\alpha^{1,0})^2=\omega^2,\quad
 \omega=2i\alpha^{1,0}\in H^0(X,K_X),\quad
 \operatorname{Im}\omega=\alpha.
\end{equation}
The marked length function determines $q$ uniquely on the fixed complex curve. In particular, the limiting differential is holomorphic and every zero has even order.
\end{proposition}

\begin{proof}
By \cite[Corollary~2.3]{CM87}, an action with a nonzero abelian length function fixes an end. Choose the Busemann function \eqref{eq:busemann} for that end. Invariance of the end gives a homomorphism $c:\Gamma\to\R$ satisfying
\begin{equation}\label{eq:busemann-period}
 b_T(\eta x)=b_T(x)+c(\eta).
\end{equation}
We verify the normalization directly. The rays $r$ and $\eta r$ have a common tail, on which $\eta r(t)=r(t+s_\eta)$ for a constant $s_\eta$. Thus $c(\eta)=-s_\eta$. At a point sufficiently far along that ray, the displacements under $\eta$ and $\eta^2$ are $|s_\eta|$ and $2|s_\eta|$. The tree identity in Lemma~\ref{lem:displacement} gives $\ell_T(\eta)=|s_\eta|=|c(\eta)|$.

Set $h=b_T\circ f$. On a regular disk this is a convex function of the real harmonic height, so it is subharmonic. The local branches differ by the constants in \eqref{eq:busemann-period}. At an ordinary zero of $q$, continuity of $f$ makes the local branch bounded and continuously extendible. At a marked point, let $P_i$ be the endpoint fixed by the peripheral element $\gamma_i$. Then
\[
 c(\gamma_i)=b_T(\gamma_i P_i)-b_T(P_i)=0.
\]
Consequently $h$ has a single-valued branch on the punctured disk. The endpoint bound makes this branch bounded, and it extends subharmonically and continuously over the puncture. We have thus extended the local branches across the omitted points before taking their differentials.

On each connected component of an overlap, any two branches differ by a constant. Their distributional differentials therefore define a global real one-form $\alpha=dh$ in the sense of currents. Choose a smooth conformal metric on $X$ and let $*$ be its Hodge star on real one-forms, with $*dx=dy$ and $*dy=-dx$ in an oriented isothermal coordinate. The current
\begin{equation}\label{eq:positive-current}
 \mu=d(*\alpha)
\end{equation}
is locally the distributional Laplacian of a subharmonic branch times area, hence is a nonnegative measure on $X$. Moreover, this measure is the exterior derivative of a global current. Pairing with the constant function $1$ gives
\[
 \mu(X)=\int_X d(*\alpha)=0.
\]
Thus $\mu=0$. By regularity for distributional harmonic functions, every branch of $h$ is smooth and harmonic. It follows that $\alpha$ is a smooth closed and co-closed real one-form on the compact curve. We use here the standard regularity and Hodge-theoretic facts for harmonic differentials \cite{For81,Ran95}.

On a regular disk write $f=\iota(u)$, where $u=\operatorname{Im}\int\sqrt q$. The restriction of $b_T$ to $\iota(u)$ is affine with slope $1$ or $-1$, except possibly for one convex corner. Since $du\ne0$, a corner inside the height range would make $b_T\circ\iota\circ u$ nonsmooth there, or equivalently would give it a positive distributional Laplacian supported on a level curve. Since the potential is harmonic, no such corner can occur. Hence $dh=du$ or $dh=-du$ locally. As
\[
 (u_z)^2\,dz^2=-\tfrac14q,
\]
we obtain \eqref{eq:abelian-square} on every regular disk. The $(1,0)$ part of a real harmonic form is holomorphic. By meromorphic continuation, \eqref{eq:abelian-square} holds on all of $X$; in particular, any possible simple pole of $q$ disappears.

The periods of $\alpha$ are exactly $c$. A real harmonic one-form on a compact curve is uniquely determined by its periods: the difference of two such forms is $dH$ for a global function $H$, and co-closedness makes $H$ harmonic, hence constant. Finally, if $c'$ is another homomorphism with $|c'|=|c|$, choose $\eta_0$ with $c(\eta_0)\ne0$ and put $c'(\eta_0)=\varepsilon c(\eta_0)$, where $\varepsilon\in\{1,-1\}$. Equality of the squared values on $\eta_0$, $\eta$, and $\eta_0\eta$ gives
\[
 c(\eta_0)c(\eta)=c'(\eta_0)c'(\eta).
\]
Thus $c'=\varepsilon c$ on all of $\Gamma$. The corresponding harmonic form changes by the same global sign, which disappears upon squaring in \eqref{eq:abelian-square}. This proves uniqueness of $q$.
\end{proof}

Figure~\ref{fig:busemann-interval} illustrates the convex corner excluded by harmonicity in this proof.

\begin{figure}[htbp]
\centering
\begin{tikzpicture}[x=1cm,y=1cm,font=\small,
  line cap=round,line join=round,>=Stealth]
  \begin{scope}
    \draw[line width=1pt] (.2,.85)--(4.7,.85);
    \draw[line width=.45pt] (.2,.76)--(.2,.94);
    \draw[line width=.45pt] (4.7,.76)--(4.7,.94);
    \draw[->,line width=.5pt] (2.45,.85)--(2.45,3.15)
      node[above=3pt] {toward the fixed end};
    \node[right=5pt] at (2.45,2.1) {$b_T$ decreases};
    \draw[fill=white,line width=.6pt] (2.45,.85) circle (2pt);
    \node[below=5pt] at (2.45,.85) {possible branch point};
    \draw[->,line width=.4pt] (.65,-.05)--(4.25,-.05)
      node[right=3pt] {$u$};
    \node at (2.45,-.55) {(a) Local image interval in $T$};
  \end{scope}
  \begin{scope}[xshift=7cm]
    \draw[->,line width=.4pt] (.1,-.05)--(4.85,-.05)
      node[right=3pt] {$u$};
    \draw[->,line width=.4pt] (.1,-.05)--(.1,3.15)
      node[above=3pt] {$b_T\circ\iota$};
    \draw[line width=.8pt] (.55,2.75)--(2.45,.85)--(4.35,2.75);
    \node[above right=6pt] at (1.35,1.95) {$-1$};
    \node[above left=6pt] at (3.55,1.95) {$+1$};
    \draw[fill=white,line width=.6pt] (2.45,.85) circle (2pt);
    \node[below=5pt] at (2.45,.85) {excluded corner};
    \node at (2.45,-.55) {(b) Convex restriction};
  \end{scope}
\end{tikzpicture}
\caption{A possible convex corner of $b_T\circ\iota$ when the ray toward the fixed end leaves the local image interval. Harmonicity of $h=b_T\circ f$ excludes an interior corner, yielding $dh=\pm du$.}
\label{fig:busemann-interval}
\end{figure}
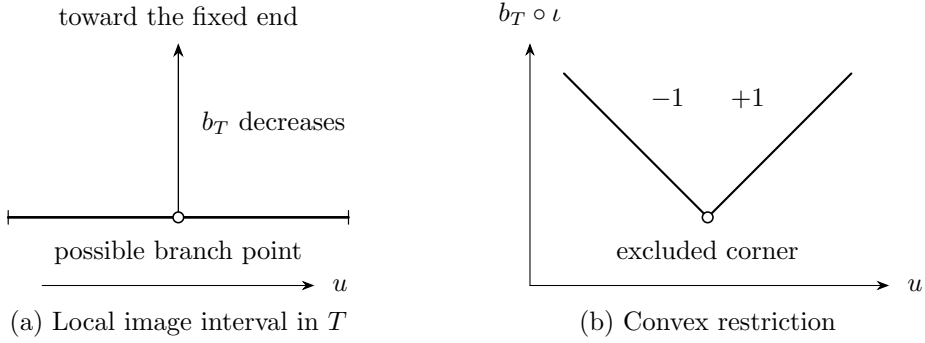

For the conjugate complex structure, the same underlying conformal metric has the opposite orientation. The Hodge star on real one-forms changes sign, so the equations $d\alpha=0$ and $d(*\alpha)=0$ are unchanged. Equal periods therefore give the same real harmonic form on the underlying surface, while its $(1,0)$ part is conjugated. Consequently, abelian limiting actions with equal marked lengths on $X$ and $\overline X$ have differentials satisfying $p=\overline q$.

 
\subsection{The comparison theorem}\label{subsec:harmonic-comparison}

We return to the accessory spaces $\A_X$ and $\A_{\overline X}$ for the conjugate marked curves. Their origins, bases, and theta characteristics are conjugate, and their topological markings agree. We now compare the limiting differentials in both the abelian and nonabelian cases.

\begin{proposition}\label{prop:conjugate-limit}
Suppose $Q_j\in\A_X$ and $R_j\in\A_{\overline X}$ have equal finite trace coordinates and
\begin{equation}\label{eq:conjugate-normalization}
 \frac{Q_j-Q_0}{T_j}\longrightarrow q,\quad
 \frac{R_j-\overline{Q_0}}{T_j}\longrightarrow p,
\end{equation}
where $T_j>0$ tends to infinity and at least one limit is nonzero. Then both limits are nonzero and
\begin{equation}\label{eq:conjugate-differential}
 p=\overline q.
\end{equation}
The same conclusion holds if equality of finite traces is replaced by equality of all marked monodromy translation lengths. For two sequences on the same complex curve the resulting identity is $p=q$.
\end{proposition}

\begin{proof}
Suppose $q\ne0$ and $p=0$. The accessory norm of the second sequence is $o(T_j)$. Lemma~\ref{lem:growth} then gives $L(\rho_{R_j}(\eta))=o(\sqrt{T_j})$ for every fixed word $\eta$. Equality of monodromy lengths and \eqref{eq:length-limit} would make the action associated with $q$ have zero length function, contradicting Lemma~\ref{lem:nonzero-length}. Interchanging the two sequences gives the same contradiction, so both limits are nonzero.

We use the common scale $1/(2\sqrt{T_j})$ and the same ultrafilter in both constructions. Equal finite traces give equal traces of all words, and hence equal hyperbolic translation lengths. By \eqref{eq:length-limit}, the two limiting actions have exactly equal nonzero marked length functions. If this function is abelian, Proposition~\ref{prop:abelian-reconstruction} and the conjugation of harmonic representatives just explained give \eqref{eq:conjugate-differential}.

If the common length function is nonabelian, Lemma~\ref{lem:minimal-tree} identifies the completed minimal trees equivariantly and places both maps in that common target. They have isometric interval descriptions on regular disks. The endpoint bounds ensure that the invariant squared distance is bounded at all marked points and zeros of either limiting differential. Real harmonic functions on $X$ and $\overline X$ are harmonic for the same underlying conformal metric, so Proposition~\ref{prop:constant-distance} applies.

To make the local comparison explicit, in a coordinate $z$ on $X$ write $q=q(z)\,dz^2$ and $p=P(w)\,dw^2$ on $\overline X$, where $w=\overline z$. The two local heights satisfy
\[
 u=\operatorname{Im}\int\sqrt q,\quad
 (u_z)^2\,dz^2=-\tfrac14q.
\]
If $v=\operatorname{Im}F(\overline z)$ with $F'(w)=\sqrt{P(w)}$, then
\[
 (v_z)^2\,dz^2=-\tfrac14\overline p.
\]
Equality of the squared derivatives gives \eqref{eq:conjugate-differential} on a common regular disk. The identity principle then extends it over $X$. When both structures are the same, both derivative formulas use that complex structure and give $p=q$.
\end{proof}

The peripheral argument excludes nonzero abelian lengths on the sphere, while the reconstruction from harmonic periods handles abelian lengths when they occur on a compact Riemann surface.


\section{Proof of the main theorems}\label{sec:finiteness}

Throughout this section, $X$ is a compact connected Riemann surface, $D$ is a finite reduced marked divisor, and $\A_X$ is the nonempty accessory space for prescribed coefficients $\boldsymbol\delta$. The group $\Gamma=\pi_1(X\setminus D,x_0)$ is fixed by the marking. We first prove the finiteness theorems and then obtain uniform counting bounds in smooth marked families. As in Section~\ref{subsec:traces}, $\K$ and $K^{\mathrm{tr}}$ denote the compact images of $\Hom(\Gamma,\SU(2))$ in the character variety and the finite-trace space, respectively.

 
\subsection{The analytic relation and unitary characters}\label{subsec:analytic}

The comparison of limiting quadratic differentials restricts the possible directions at infinity. We first show how this restriction implies discreteness.

\begin{lemma}\label{lem:psh}
Let $A\subset\C^d$ be a closed complex analytic set. Suppose every limit of normalized points $x_j/\norm{x_j}$, where $x_j\in A$ and $\norm{x_j}\to\infty$, lies in a fixed totally real linear subspace of real dimension $d$. Then $A$ is discrete.
\end{lemma}

\begin{proof}
A complex-linear change of coordinates sends the given real subspace to $\R^d$. In the resulting coordinates the hypothesis implies
\begin{equation}\label{eq:real-directions}
 \frac{|\operatorname{Im}\xi|}{|\xi|}\longrightarrow0
 \quad\text{as }|\xi|\to\infty,\quad \xi\in A.
\end{equation}
Otherwise, a sequence violating this limit would have a convergent normalized subsequence with a nonreal limit.

Fix $0<\varepsilon<1$ and put
\begin{equation}\label{eq:psh-function}
 B(\xi)=|\operatorname{Im}\xi|^2-\varepsilon|\operatorname{Re}\xi|^2.
\end{equation}
Its Levi matrix is $(1-\varepsilon)I_d/2$, so $B$ is strictly plurisubharmonic. Since $B=(1+\varepsilon)|\operatorname{Im}\xi|^2-\varepsilon|\xi|^2$, equation~\eqref{eq:real-directions} gives $B(\xi)\le-\varepsilon|\xi|^2/2$ on $A$ outside a large ball.

Suppose $C$ is a positive-dimensional irreducible component of $A$. The component is closed. For any $x\in C$, its nonempty superlevel set $\{y\in C:B(y)\ge B(x)\}$ is compact, so $B|_C$ has a maximum at some $x_*$. A positive-dimensional analytic germ contains a nonconstant holomorphic curve germ through its basepoint: take a one-dimensional analytic section and normalize a branch \cite{Chi89}. Thus there is a nonconstant holomorphic disk $h$ in $C$ with $h(0)=x_*$. The subharmonic function $B\circ h$ has an interior maximum and is constant. But
\[
 \frac{\partial^2(B\circ h)}{\partial t\,\partial\overline t}
 =\frac{1-\varepsilon}{2}\sum_{\nu=1}^d|h_\nu'(t)|^2,
\]
which would force $h$ to be constant. We conclude that no positive-dimensional components exist. A zero-dimensional analytic set is locally discrete \cite{Chi89}.
\end{proof}

We use conjugate origins and bases to identify $\A_X\times\A_{\overline X}$ with $\C^m\times\C^m$, where $m=\dim H^0(X,K_X^{\otimes2}(D))$. For both Riemann surfaces, we use the identity on the underlying topological surface as the marking.

\begin{definition}[The trace relation]\label{def:trace-relation}
The trace relation between the conjugate accessory spaces is
\begin{equation}\label{eq:trace-relation}
 \widetilde Z=\{(Q,R)\in\A_X\times\A_{\overline X}:F_X(Q)=F_{\overline X}(R)\}.
\end{equation}
As in \eqref{eq:intro-relation}, the equal-character relation $Z$ consists of the pairs satisfying $\Hol_X(Q)=\Hol_{\overline X}(R)$.
\end{definition}

Both relations are closed complex analytic. Indeed, the first is the common zero set of finitely many entire functions, while the second is the inverse image of the diagonal in the affine character variety under the holonomy maps. In particular, $Z\subset\widetilde Z$.

\begin{proposition}\label{prop:discrete-relation}
The sets $\widetilde Z$ and $Z$ are closed and discrete.
\end{proposition}

\begin{proof}
The case $m=0$ is immediate. Otherwise, let $(Q_j,R_j)\in\widetilde Z$ escape every compact subset and set
\[
 T_j=\max\{1,\norm{Q_j-Q_0},\norm{R_j-\overline{Q_0}}\}.
\]
A subsequence of the normalized differences converges to $(q,p)$ with $\max\{\norm q,\norm p\}=1$. Proposition~\ref{prop:conjugate-limit} gives $p=\overline q$. The conclusion holds for every convergent normalized subsequence. Thus all directions at infinity lie in the totally real subspace
\begin{equation}\label{eq:real-subspace}
 L_{\R}=\{(a,b)\in\C^m\times\C^m:b=\overline a\}.
\end{equation}
The complex-linear coordinates
\[
 z=\frac{a+b}{2},\quad w=\frac{a-b}{2i}
\]
send $L_{\R}$ to $\R^{2m}$. By equivalence of norms, the normalization by $T_j$ is equivalent to Euclidean normalization. We can therefore apply Lemma~\ref{lem:psh} to $\widetilde Z$. Its closed analytic subset $Z$ is also discrete.
\end{proof}

\begin{proof}[Proof of Theorem~\ref{thm:holonomy}]
By Proposition~\ref{prop:properness}, the holonomy map is proper, and by Proposition~\ref{prop:discrete-relation}, $Z$ is discrete. To prove the final assertion, we consider the larger preimage under the finite-trace map
\[
 E^{\mathrm{tr}}=F_X^{-1}(K^{\mathrm{tr}}).
\]
This set is compact by properness. Every coordinate of $K^{\mathrm{tr}}$ is real, and coefficient conjugation satisfies \eqref{eq:conjugate-traces}. Hence
\[
 \{(Q,\overline Q):Q\in E^{\mathrm{tr}}\}\subset\widetilde Z.
\]
The graph on the left is compact and the set on the right is closed discrete, so the graph is finite. Projecting to the first factor identifies this graph with $E^{\mathrm{tr}}$. Finally $\Hol_X^{-1}(\K)\subset E^{\mathrm{tr}}$, which proves the theorem.
\end{proof}

\begin{proof}[Proof of Theorem~\ref{thm:unitary}]
If the accessory space is empty, the conclusion is immediate. Otherwise, put $\delta_i=(1-\beta_i^2)/4$ and use the scalar lift defined by a theta characteristic. For a connection with unitarizable projective monodromy, we lift a projective conjugating element to $\SL_2(\C)$. Each conjugated monodromy matrix belongs to the inverse image of $\PSU(2)$ under $\SL_2(\C)\to\PSL_2(\C)$. The inverse image is $\SU(2)$: a determinant-one matrix projectively equal to $U\in\SU(2)$ is $cU$ with $c^2=1$. Its character lies in $\K$, and Theorem~\ref{thm:holonomy} gives finiteness.
\end{proof}

We next formulate discreteness in the character variety itself. Let $\mathcal Y_X=\Hol_X(\A_X)$ and $\mathcal Y_{\overline X}=\Hol_{\overline X}(\A_{\overline X})$, with the reduced analytic structures of Corollary~\ref{cor:finite-holonomy}.

\begin{corollary}\label{cor:image-intersection}
The intersection $\mathcal Y_X\cap\mathcal Y_{\overline X}$ is a closed discrete analytic subset of $\X(\Gamma)$. Every compact subset of $\X(\Gamma)$ contains only finitely many common characters and has only finitely many pairs in $Z$ above it.
\end{corollary}

\begin{proof}
For a compact $K\subset\X(\Gamma)$, the pairs with common character in $K$ form the closed set
\[
 Z_K=Z\cap\bigl(\Hol_X^{-1}(K)\times\Hol_{\overline X}^{-1}(K)\bigr).
\]
By properness, this set is compact; by Proposition~\ref{prop:discrete-relation}, it is finite. The image of this set is exactly $K\cap\mathcal Y_X\cap\mathcal Y_{\overline X}$. Both holonomy images are closed analytic. Their intersection is therefore closed analytic. Since it meets each compact neighborhood in only finitely many points, it is discrete.
\end{proof}

For a common character $\chi$, the corresponding set of pairs is $\Hol_X^{-1}(\chi)\times\Hol_{\overline X}^{-1}(\chi)$, and its cardinality is the product of the two fiber cardinalities. Coefficient conjugation also embeds $F_X^{-1}(\R^N)$ into $\widetilde Z$ as a graph. It follows that the accessory set with real traces is closed and discrete, and imposing bounded trace conditions gives finite subsets.

\begin{remark}[Real traces]\label{rem:real-traces}
The set of connections with real traces need not be finite. For a rectangular torus $\C/(a\Z+ib\Z)$, with $a,b>0$, let $z$ be its degree-two map to the sphere and write $dw/dz=1/\sqrt{P(z)}$ in an affine coordinate, where $P$ has the three finite branch values as simple roots. The developing maps $f_k(w)=\exp(2\pi k w/b)$ give sphere connections
\[
 Q_k(z)=-\frac{(2\pi k/b)^2}{4P(z)}+\frac12\mathcal S(w)(z),\quad k=1,2,\ldots.
\]
All four exponent differences are $1/2$. The torus periods act by a positive dilation and the identity, while the sheet involution acts by $f\mapsto1/f$. Conjugation by $f\mapsto if$ changes inversion to $f\mapsto-1/f$, so the projective monodromy lies in $\PSL_2(\R)$ and all scalar traces are real. The connections are distinct and unbounded. Thus discreteness of the real-trace set alone does not imply finiteness of the compact unitary locus.
\end{remark}

\begin{remark}\label{rem:resonance}
At an integral exponent difference, the logarithmic Frobenius term and local Jordan form can vary within the accessory space. The estimates above are made along fixed paths outside $D$, while the bounds near punctures concern the limiting quadratic differential with at most simple poles. The preimage of the compact character set may contain connections with nonsemisimple monodromy. For example, on an unmarked torus the connection $Q=0$ with the trivial theta characteristic has unipotent monodromy and trivial character. Its projective monodromy is not unitarizable.
\end{remark}

If $D$ is empty and $g(X)\ge1$, the unitarizable locus itself is empty. A holomorphic projective connection has an unbranched developing map, whose pullback of the round metric would descend to a smooth curvature-one metric on $X$. Gauss--Bonnet would give its positive area as $2\pi(2-2g(X))\le0$. The torus example in Remark~\ref{rem:resonance} also illustrates the distinction between a unitary character and an invariant spherical metric.

 
\subsection{Uniform counting in smooth families}\label{subsec:uniform}

Consider the local holomorphic family $\pi:\mathcal C\to B$ of Section~\ref{subsec:relative}, with disjoint marked sections, fixed coefficients $\boldsymbol\delta$, a compatible theta characteristic, and the affine chart \eqref{eq:relative-accessory}. We write $X_b=\pi^{-1}(b)$. The topological marking identifies the fundamental groups with a fixed $\Gamma$, so $\K\subset\X(\Gamma)$ and its finite-trace image $K^{\mathrm{tr}}\subset\R^N$ are independent of $b$.

\begin{corollary}\label{cor:uniform-count}
For every compact set $K\Subset B$, there is a finite integer $N_K$ such that
\begin{equation}\label{eq:uniform-count}
 \#\Hol_{X_b}^{-1}(\K)\le N_K\quad(b\in K).
\end{equation}
The same bound holds for the number of connections with unitarizable projective monodromy on each fiber.
\end{corollary}

\begin{proof}
We apply the standard uniform bound for the number of connected components of fibers of a definable set in an o-minimal structure \cite[Section~4.4]{vdDM96}. In our setting, the structure is $\R_{\mathrm{an}}$, generated by real-analytic functions restricted to compact boxes on whose neighborhoods they are analytic \cite[Section~2]{vdDM96}.

The unitary representation space is semialgebraic in the real and imaginary parts of a finite set of generator matrices: the unitary equations, determinant equations, and finitely many group relations are polynomial equations. The finite trace functions are also polynomial in these entries. By the projection theorem for semialgebraic sets, their compact image $K^{\mathrm{tr}}$ is semialgebraic.

For a closed real coordinate box $B_0$ contained in a slightly larger parameter chart, relative properness of the finite-trace map makes
\[
 \{(b,Q):b\in B_0,\ F_{X_b}(Q)\in K^{\mathrm{tr}}\}
\]
compact. Consequently, in the affine trivialization, the coefficient vectors belong to a fixed compact box. The jointly holomorphic finite-trace functions are real analytic on a neighborhood of the resulting product box. It follows that the condition of belonging to $K^{\mathrm{tr}}$ defines an $\R_{\mathrm{an}}$-definable set there. Its fiber at each $b$ is finite by the proof of Theorem~\ref{thm:holonomy}. Because these fibers are finite, the bound on the number of components also bounds their cardinalities uniformly. Finitely many such boxes cover $K$. Taking the maximum of these bounds gives a uniform bound for every finite-trace preimage and hence for every $\Hol_{X_b}^{-1}(\K)$. The last assertion follows from the lift used in the proof of Theorem~\ref{thm:unitary}.
\end{proof}

The geometric estimate in Proposition~\ref{prop:relative-properness} allows us to restrict the coefficients to a compact box. We then combine real analyticity with finiteness on the fixed fibers to obtain the counting bound. In the same family, relative properness and the argument for compact analytic fibers show that the relative holonomy map is finite onto its reduced analytic image; its fiber cardinalities over individual base-character pairs are locally bounded by the coherent direct-image algebra.


\section{Geometric consequences}\label{sec:geometric}

We now apply the finiteness theorem to spherical metrics with conical singularities, ramification divisors, and meromorphic functions.

 
\subsection{Spherical conical metrics}\label{subsec:spherical}

Let $X$ have genus $g$, and fix $D=\{a_1,\ldots,a_n\}$ and positive real parameters $\boldsymbol\beta=(\beta_1,\ldots,\beta_n)$. For the correspondence with metrics, we use the coefficients $\delta_i=(1-\beta_i^2)/4$. We first describe the correspondence and the metrics within a projective class, and then give conditions excluding coaxial monodromy.

\begin{definition}[Spherical conical metrics]\label{def:spherical}
Let $\Sph(X,D,\boldsymbol\beta)$ be the set of conformal metrics of curvature one on $X\setminus D$ which, in a coordinate centered at $a_i$, have the form
\begin{equation}\label{eq:cone-metric}
 h=e^{2v_i(z)}|z|^{2\beta_i-2}|dz|^2,
\end{equation}
where $v_i$ is real and continuous at $0$ and smooth off $0$. A developing map is an orientation-preserving local isometry from $\widetilde{X\setminus D}$ to the round sphere. Two metrics are \emph{projectively equivalent} if their developing maps $f_1,f_2$ satisfy $f_2=M\circ f_1$ for a M\"obius transformation $M$ of the target sphere; denote this relation by $\sim$. A metric has \emph{coaxial monodromy} if its projective monodromy is conjugate into the rotations $w\mapsto e^{it}w$ fixing $0$ and $\infty$ individually.
\end{definition}

We compare metrics on the fixed domain $X$ using the topology of $C^2$ convergence of conformal densities on compact subsets of $X\setminus D$. Since two developing maps of the same metric differ by a round rotation, the relation $\sim$ does not depend on that choice. A mark with $\beta_i=1$ is a removable ordinary point. We distinguish noncoaxial, nontrivial coaxial, and trivial monodromy, as in the standard description of the metrics within a projective class \cite[Section~2]{Ere20,Ere24}.

\begin{proposition}\label{prop:metric-classes}
The quotient $\Sph(X,D,\boldsymbol\beta)/{\sim}$ is finite. Each class is homeomorphic to a point, to $\R_{>0}$, or to $\HH^3$, according as its projective monodromy is noncoaxial, nontrivial coaxial, or trivial. In particular, all noncoaxial metrics with the fixed data form a finite set. Every continuous family parametrized by a connected space has constant Schwarzian and lies in one of these classes.
\end{proposition}

\begin{proof}
The correspondence with developing maps associates to each metric a projective connection $Q_h=\tfrac12\mathcal S(f)$ with unitarizable projective monodromy \cite[Section~2]{CWW+15}\cite[Sections~2 and~4]{Ere24}. In a local conformal expression $h=e^{2u}|dz|^2$, its coefficient is $u_{zz}-u_z^2$, with coefficient $(1-\beta_i^2)/4$ of the $z^{-2}$ term at $a_i$. Away from $D$, local univalence makes the Schwarzian holomorphic, including at poles of an affine target coordinate. The standard correspondence also gives the required local regularity at each cone and the corresponding regular-singular expansion.

Conversely, take a connection with the prescribed coefficients and unitarizable projective monodromy. Choose a solution ratio $f$ whose monodromy is in $\PSU(2)$. The nonzero Wronskian ensures local univalence outside $D$. At a nonintegral difference $\beta_i>0$, Frobenius theory gives $f=z^{\beta_i}a(z)$, where $a$ is holomorphic and $a(0)\ne0$, after a round rotation sending the local fixed axis to $0,\infty$. If necessary, we interchange the two axis points to obtain the positive exponent. At a positive integral difference $\beta_i=k$, a nonzero logarithmic term would produce nontrivial parabolic projective monodromy. Unitary monodromy excludes this term, so the ratio extends meromorphically with local degree $k$. The pullback of the round metric then has precisely the form \eqref{eq:cone-metric}. For $k=1$ the metric extends smoothly and positively. These are the local models used in \cite{CWW+15,Ere20c,Ere24}.

Two developing maps have the same Schwarzian if and only if they differ by a constant M\"obius transformation. Indeed, the local bases of solutions of $y''+Q_z y=0$ differ by a constant invertible matrix, and the same transformation between their ratios extends to the universal cover by analytic continuation. Hence the projective classes are in bijection with the set in Theorem~\ref{thm:unitary}, and the quotient is finite.

Fix one class and a developing map $f$ with scalar unitary lift $\rho$ on $X\setminus D$. Write every other developing map with this Schwarzian as $M\circ f$ and choose a determinant-one representative of $M$. The corresponding positive Hermitian form $H=M^*M$ defines a single-valued metric precisely when
\begin{equation}\label{eq:invariant-forms}
 H=H^*>0,\quad \det H=1,\quad
 \rho(\eta)^*H\rho(\eta)=H\quad(\eta\in\Gamma).
\end{equation}
Conversely such a form is realized by $M=H^{1/2}$. With $\mathbf v_f=(f,1)^{\mathsf t}$, its metric is
\begin{equation}\label{eq:hermitian-metric}
 h_H=\frac{4|f'|^2}{(\mathbf v_f^*H\mathbf v_f)^2}|dz|^2.
\end{equation}
These metrics all have the same local angles. Distinct forms determine distinct metrics: equality in \eqref{eq:hermitian-metric} implies equality of positive Hermitian quadratic polynomials on the open image of a univalent chart, hence equality of their coefficients.

Since $\rho$ is unitary, invariance in \eqref{eq:invariant-forms} means that $H$ commutes with all its matrices. Noncoaxial projective monodromy is equivalent to the absence of a common invariant complex line for this unitary lift. A nonscalar $H$ would have invariant eigenspaces; hence $H=I$ in this case. For nontrivial coaxial monodromy, an axis basis diagonalizes the representation and some matrix has distinct eigenvalues. The invariant forms are precisely $\operatorname{diag}(t,t^{-1})$, $t>0$. For trivial projective monodromy, all scalar matrices are $\pm I$, so every positive determinant-one Hermitian form occurs. The space of these forms is the usual model of $\HH^3$.

The parametrizations are homeomorphisms. The forward map is smooth on compact subsets outside $D$. For the inverse map, we choose four regular points whose developing values give independent evaluations on the real four-dimensional space of Hermitian quadratic forms. Such points exist because a Hermitian polynomial vanishing on an open set is zero. We recover these evaluations from the metric by taking positive square roots in \eqref{eq:hermitian-metric}, and then recover $H$ continuously by solving an invertible real linear system. Finally $h\mapsto Q_h=u_{zz}-u_z^2$ is continuous in the stated topology. Since its image is a finite subset of the accessory space, $Q_h$ is constant along any connected continuous family.
\end{proof}

We next give an obstruction to coaxial monodromy that applies in every genus. We use the character differential of a coaxial metric, as studied in \cite[Theorems~1.4--1.5]{CWW+15}. On the sphere the resulting conditions are the necessary part of Eremenko's coaxial angle theorem \cite[Theorem~1]{Ere20c}; see also \cite{WWX26} for the geometry of reducible metrics.

\begin{proposition}\label{prop:coaxial-obstruction}
Put $J=\{i:\beta_i\notin\Z\}$ and
\begin{equation}\label{eq:coaxial-bound}
 B_g=2-2g-|J|+\sum_{i\notin J}(\beta_i-1).
\end{equation}
If a metric in $\Sph(X,D,\boldsymbol\beta)$ is coaxial, there are signs $\varepsilon_i\in\{1,-1\}$, $i\in J$, for which
\begin{equation}\label{eq:coaxial-signs}
 S_\varepsilon:=\sum_{i\in J}\varepsilon_i\beta_i\in\Z,\quad
 |S_\varepsilon|\le B_g,\quad B_g-S_\varepsilon\in2\Z.
\end{equation}
Consequently, failure of these conditions for every choice of signs implies finiteness of the full metric set.
\end{proposition}

\begin{proof}
We normalize the developing map so that monodromy acts by rotations about $0,\infty$. Then $\omega=df/f$ descends to a nonzero meromorphic one-form on $X$. At a nonintegral cone it has a simple pole with residue $\pm\beta_i$. At an integral cone it has a simple pole with residue $\pm\beta_i$ if the developing value is an axis point, and otherwise it has a zero of order $\beta_i-1$. For an ordinary marked point, this order is zero. At an unmarked smooth point, zeros and poles of $f$ give simple poles of $\omega$ with residues $\pm1$. By local univalence outside $D$, there are no other zeros or poles.

Let $I$ be the complement of $J$, let $A\subset I$ be the integer marks giving poles, and let $k$ be the number of unmarked poles. The divisor degree $\deg\divisor(\omega)=2g-2$ gives
\[
 \sum_{i\in I\setminus A}(\beta_i-1)=|J|+|A|+k+2g-2,
 \quad B_g=k+\sum_{i\in A}\beta_i.
\]
The residue theorem gives signs $\eta_j\in\{1,-1\}$ such that
\[
 S_\varepsilon=-\sum_{i\in A}\varepsilon_i\beta_i-\sum_{j=1}^k\eta_j.
\]
The right-hand side is an integer whose absolute value is at most $B_g$ and whose parity agrees with that of $B_g$. This proves \eqref{eq:coaxial-signs}. If no choice of signs satisfies these conditions, every metric is noncoaxial, and the final assertion follows from Proposition~\ref{prop:metric-classes}.
\end{proof}

\begin{corollary}\label{cor:nonintegral-metrics}
If every $\beta_i$ is positive and nonintegral and $2g-2+n>0$, then $\Sph(X,D,\boldsymbol\beta)$ is finite.
\end{corollary}

\begin{proof}
Here $B_g=2-2g-n<0$, so Proposition~\ref{prop:coaxial-obstruction} excludes every coaxial metric. The conclusion now follows from Proposition~\ref{prop:metric-classes}.
\end{proof}

More generally, the full metric set is finite if and only if it contains no coaxial metric, because each realized coaxial class contains an infinite family. For fixed nontrivial coaxial monodromy, the character differential is unique up to sign within its projective class: the unordered pair of common fixed points is unique, and changes of an axis coordinate have the form $f\mapsto cf$ or $f\mapsto c/f$. They change $df/f$ to itself or its negative. Thus for fixed $X,D,\boldsymbol\beta$ there are only finitely many such character differentials modulo sign.

For the converse construction, we require both the complete divisor pattern described above and purely imaginary periods. Given a meromorphic $\omega$ with exactly those zeros and poles, their indicated real residues, and $\int_\eta\omega\in i\R$ for every closed path avoiding the poles, the function $f=c\exp\int\omega$ has rotational monodromy and yields exactly the prescribed metric. A residue $\pm1$ at an auxiliary unmarked pole gives local degree one, hence a smooth point. Nontrivial monodromy means that at least one period is not in $2\pi i\Z$. This construction and the complete zero and pole requirements are the correspondence of \cite[Theorems~1.4--1.5]{CWW+15}. The condition of nontrivial monodromy is relevant to the finiteness of the differential: when monodromy is trivial, target M\"obius changes can alter $df/f$ continuously while leaving the Schwarzian fixed. On $\PP^1$, for instance, $f_a(z)=z-a$ has constant Schwarzian zero, whereas $df_a/f_a=dz/(z-a)$ varies with $a$.

\begin{remark}[Sphere cases]\label{rem:coaxial-angles}
At $g=0$, \eqref{eq:coaxial-bound}--\eqref{eq:coaxial-signs} recover the original signed-sum and parity exclusions. The developing maps
\[
 f_t(z)=t[z(z-1)]^{1/3},\quad t>0,
\]
give distinct metrics with fixed angle parameters $(1/3,1/3,2,2/3)$ at $(0,1,1/2,\infty)$. Their monodromy consists of rotations by cube roots of unity. They have the same Schwarzian, but their leading metric coefficient at $0$ is $4t^2/9$. This example shows that three nonintegral angles can coexist with a coaxial family when an integral angle is present.

Ignoring ordinary marked points, a smooth sphere has one projective class with metric space $\HH^3$, whereas a sphere with one genuine conical point admits no such metric. For two genuine points, choose coordinates $0,\infty$. The angle parameters must agree, say with value $\beta$, and $Q=(1-\beta^2)/(4z^2)$. The map $z^\beta$ gives the class; it is a positive real ray for nonintegral $\beta$, and $\HH^3$ for integral $\beta\ge2$. More generally, on the sphere one nonintegral angle with all others integral is impossible, while two nonintegral angles force coaxial monodromy: after the integral peripheral monodromies become identity, the remaining two are inverse unitary rotations. On positive-genus curves, all-ordinary marked data give the empty set by Gauss--Bonnet. For the sphere with at least three genuine points and all $0<\beta_i<1$, the stronger uniqueness assertion is already known, see \cite{LT92} and \cite[Theorem~1.9 (ii)]{BGJ+19}.
\end{remark}

We can also express the metric conclusion in terms of the singular Liouville equation. Let $h_0$ be a smooth conformal background metric with area $A_0$ and curvature $K_0$. Use $\Delta_0=\operatorname{div}_{h_0}\operatorname{grad}$, whose nonconstant eigenvalues are negative, and let $\delta_{a_i}^{(0)}$ denote the Dirac measure relative to its area form. The correspondence $h=e^{2u}h_0$ identifies the metric set with solutions, having the prescribed continuous conical remainders, of
\begin{equation}\label{eq:liouville-metric}
 \Delta_0u=K_0-e^{2u}+2\pi\sum_i(\beta_i-1)\delta_{a_i}^{(0)}.
\end{equation}
Away from $D$, this identity follows from the curvature formula, while $\Delta\log|z|=2\pi\delta_0$ accounts for the singular terms. Integrating yields
\begin{equation}\label{eq:area-constraint}
 A_{\boldsymbol\beta}:=\int_Xe^{2u}\,dA_0
 =2\pi\left(2-2g+\sum_i(\beta_i-1)\right)>0.
\end{equation}
Consequently, the set is empty if $A_{\boldsymbol\beta}$ is nonpositive.

For comparison with the normalized mean-field equation, choose $h_0$ to have constant curvature and take mean-zero Green functions $G_i$ satisfying $\Delta_0G_i=\delta_{a_i}^{(0)}-1/A_0$. Put $\sigma=2\pi\sum_i(\beta_i-1)G_i$ and $a_{\boldsymbol\beta}=e^{2\sigma}$. Equation~\eqref{eq:liouville-metric} is equivalent to
\begin{equation}\label{eq:mean-field}
 \Delta_0v+\rho\left(\frac{a_{\boldsymbol\beta}e^v}{\int_Xa_{\boldsymbol\beta}e^v\,dA_0}-\frac1{A_0}\right)=0,
 \quad \int_Xv\,dA_0=0,\quad \rho=2A_{\boldsymbol\beta}.
\end{equation}
Indeed $v=2(u-\sigma)-A_0^{-1}\int_X2(u-\sigma)\,dA_0$, and the inverse formula is
\[
 u=\sigma+\frac v2+\frac12\log\frac{A_{\boldsymbol\beta}}{\int_Xa_{\boldsymbol\beta}e^v\,dA_0}.
\]
These formulas give a bijection for solutions with bounded conical remainders. The finiteness assertions for metrics therefore apply to the corresponding solutions of \eqref{eq:mean-field}; the continuous coaxial freedoms are exactly those in Proposition~\ref{prop:metric-classes}. These normalizations agree with the curvature and mean-field formulations in \cite{Tro91,BDM+11,BGJ+19}.

 
\subsection{Ramification divisors and stable extensions}\label{subsec:extensions}

Let $R=\sum_i m_i a_i$ be an effective integral divisor on $X$, where $m_i>0$ on its support. Write $\Sph(X,R)$ for $\Sph(X,\operatorname{supp}R,\boldsymbol\beta)$ with $\beta_i=m_i+1$. Thus $R$ is the ramification divisor of a developing map, while $D=\operatorname{supp}R$ is the reduced marked set. We use this formulation to relate metric finiteness to the correspondence with stable extensions established by Li--Song--Xu \cite{LSX21}.

\begin{corollary}\label{cor:odd-divisors}
Suppose $g=g(X)\ge1$ and $R$ has odd degree $d>2g-2$. Then
\[
 0<\#\Sph(X,R)<\infty.
\]
\end{corollary}

\begin{proof}
All angle parameters are integral. A coaxial metric would have $J=\varnothing$ in Proposition~\ref{prop:coaxial-obstruction}, so $S_\varepsilon=0$ and $B_g=2-2g+d$ would be even. Since $d$ is odd, this is impossible. Thus every metric is noncoaxial, and finiteness follows from Proposition~\ref{prop:metric-classes}.

To prove nonemptiness, we apply existing results on the existence of curvature-one metrics with the curve and every point of $R$ fixed. Put $a=2-2g+d$, a positive odd integer. If $a=1$, Troyanov's Theorem~C \cite{Tro91} applies: its area quantity is $2\pi a=2\pi$, which is strictly less than both $4\pi$ and twice the smallest cone angle $2\pi(m_i+1)\ge4\pi$. If $a\ge3$, the supercritical existence theorem of Bartolucci--De Marchis--Malchiodi applies \cite[Theorem~1.1]{BDM+11}; see also its original author formulation in \cite[Theorems~0.5 and~0.7]{DM10}. In its notation the singularity orders are $m_i>0$, the prescribed curvature is $1$, and the parameter is $\rho=4\pi a>8\pi$. Its exceptional set is
\[
 \left\{8\pi k+8\pi\sum_i n_i(1+m_i)>0:
 k\in\Z_{\ge0},\ n_i\in\{0,1\}\right\}.
\]
Every member is an integral multiple of $8\pi$, whereas $a$ is odd. The nonresonance requirement of that existence theorem is satisfied. Together, these two cases establish nonemptiness.
\end{proof}

Li--Song--Xu proved the conclusion for almost every effective divisor of the indicated odd degree \cite[Corollary~2.3]{LSX21} and discussed whether this restriction could be removed in their Section~6. Our Corollary~\ref{cor:odd-divisors} establishes finiteness for every such divisor. On elliptic curves, this conclusion already follows from Wang's algebraic treatment of the multisource Liouville equation \cite[Theorems~0.1 and~3.5, Corollary~4.2]{Wang21}. In that case, the algebraic length $\tfrac12\prod_i(m_i+1)$ bounds the number of distinct metrics, although special fibers may have multiplicities. Since $m_i+1\le2^{m_i}$, this also gives the bound $2^{d-1}$ for all elliptic curves and all marked configurations of the same odd total degree $d$. On the sphere an odd ramification divisor gives the empty set, since an integral-angle developing map is rational and its total ramification is even.

We distinguish the scalar lift on the punctured surface from a lift on the closed Riemann surface. For integral angles the local projective monodromy is trivial. The scalar lift from Section~\ref{subsec:scalar} has peripheral monodromy $(-1)^{m_i}I$, since its local exponents are $(1\pm(m_i+1))/2$. If $A_j,B_j$ are lifts of the handle monodromies, the surface relation becomes
\begin{equation}\label{eq:odd-lift}
 \prod_{j=1}^g[A_j,B_j]=(-1)^dI.
\end{equation}
Changing signs of the handle lifts leaves this product unchanged. Thus the projective representation extends over $X$ but has a closed-curve $\SU(2)$ lift exactly when $d$ is even. The bundle interpretation in \cite{LSX21} gives the same parity criterion. On an elliptic curve with odd $d$, the lifts anticommute. Their traces vanish and their squares are $-I$, so their projective images generate the Klein four group of half-turns about perpendicular axes. Such monodromy is noncoaxial even though the projective group is commutative.

We now apply the finiteness of metrics to the ramification divisor map. Fix holomorphic line bundles $L_1,L_2$ on $X$ and put $k=\deg L_2-\deg L_1>0$ and
\[
 B_{12}=K_X\otimes L_1^{-1}\otimes L_2.
\]
A rank-two holomorphic bundle $E$ is stable if each of its line subbundles has degree less than $\deg E/2$. We work with projectivized extension classes, so the equivalence includes scalar changes in the identifications of the two endpoint line bundles.

\begin{definition}[The ramification divisor map]\label{def:ramification-map}
Let
\[
 U_{12}=\PP H^1(X,L_1\otimes L_2^{-1})^s
\]
be the open subset of classes of extensions $0\to L_1\to E\to L_2\to0$ with stable middle bundle. Its Hermitian--Einstein projective connection defines a noncoaxial metric, and the derivative of the corresponding section of the projective bundle $\PP(E)$ of lines in $E$ is a nonzero section $\theta_E$ of $B_{12}$. The real-analytic map
\begin{equation}\label{eq:ramification-map}
 \mathscr R_{L_1,L_2}:U_{12}\longrightarrow\PP H^0(X,B_{12}),
 \quad [E]\longmapsto[\theta_E]
\end{equation}
is the ramification divisor map of \cite[Theorems~2.3--2.4]{LSX21}.
\end{definition}

The projective class $[\theta_E]$ determines the zero divisor of its representative. A noncoaxial metric with integral angles comes from a stable extension by \cite[Theorem~2.1]{LSX21}. Tensoring the entire extension by a line bundle leaves the construction unchanged. On the fixed space $U_{12}$, the metric determines the projective extension class uniquely by \cite[Lemma~3.1]{FSX26}.

\begin{proposition}\label{prop:ramification-fibers}
Every fiber of \eqref{eq:ramification-map} is finite. For a nonzero section $\theta\in H^0(X,B_{12})$ with divisor $R$, the fiber is in bijection with the noncoaxial metrics in $\Sph(X,R)$. On each nonempty connected component of $U_{12}$, the differential of $\mathscr R_{L_1,L_2}$ has full real rank on an open dense subset. Its critical set has measure zero, and its image contains a nonempty open subset of $\PP H^0(X,B_{12})$.
\end{proposition}

\begin{proof}
Every extension in the fiber over $[\theta]$ gives a noncoaxial metric with ramification $R$. By the injectivity for fixed line bundles cited above, distinct points of this fiber correspond to distinct metrics. This set is finite by Proposition~\ref{prop:metric-classes}.

To prove the claimed bijection, we start with any noncoaxial metric with ramification $R$. By the stable-extension correspondence it is represented by $0\to L\to E\to M\to0$. Its ramification section gives
\[
 K_X\otimes L^{-1}\otimes M\simeq\mathcal O_X(R)\simeq B_{12}.
\]
We tensor the extension by $N=L_1\otimes L^{-1}$. The subline becomes $L_1$ and the quotient becomes
\[
 M\otimes N\simeq L_1\otimes K_X^{-1}\otimes B_{12}\simeq L_2.
\]
Tensoring preserves stability and the associated projective metric. After choosing identifications of the endpoint line bundles, we obtain a point of $U_{12}$. Its ramification section has divisor $R$ and hence differs from $\theta$ by a nonzero constant on the compact connected curve. This proves surjectivity onto the metric set, while injectivity for fixed line bundles gives the claimed bijection. Moreover, tensoring does not introduce any additional equivalence within $U_{12}$: an equivalence preserving its subline would require $L_1\otimes N\simeq L_1$, hence $N\simeq\mathcal O_X$, and the remaining endpoint changes only scale the extension class.

By Serre duality,
\[
 h^1(X,L_1\otimes L_2^{-1})=h^0(X,B_{12})=k+g-1,
\]
where the second equality follows from Riemann--Roch and $\deg(L_1\otimes L_2^{-1})<0$. Thus the source and target in \eqref{eq:ramification-map} have the same real dimension $e=2(k+g-2)$ whenever the source is nonempty. On each component of the source, let $r$ denote the maximal rank of the differential. If $r<e$, a point where the rank equals $r$ has a neighborhood on which the rank is constant and equal to $r$: a nonvanishing $r$-minor gives the lower bound and maximality gives the upper bound. The constant-rank theorem would then produce local fibers of positive dimension $e-r$, contradicting finiteness. We conclude that the maximal rank is $e$.

The analytic section $\bigwedge^e d\mathscr R_{L_1,L_2}$ is not identically zero on the component. Analytic continuation shows that its zero set has empty interior, and the standard zero-set theorem for real-analytic functions makes it measure zero \cite{KP02}. Consequently, the locus of full rank is open and dense. By the inverse function theorem, the image contains an open subset. If $e=0$, full rank holds everywhere and the same assertion is immediate.
\end{proof}

The stable locus is nonempty for $g\ge2$ and $k>0$, and for $g=1$ precisely when $k$ is odd and positive \cite[Theorem~2.2]{LSX21}. Stable rank-two bundles on the sphere do not occur. The conclusion about full rank addresses the question concerning rank and open images discussed in \cite[Section~6]{LSX21}.

When $g\ge1$ and $k$ is odd, every target divisor has degree $2g-2+k>2g-2$. Corollary~\ref{cor:odd-divisors} and the fiber identification imply that \eqref{eq:ramification-map} is surjective with nonempty finite fibers. Existence and surjectivity were already established in the theory of stable extensions \cite{LSX21}. For $k=1$, every nonzero extension is stable \cite[Theorem~2.2]{LSX21}. The source is then the compact projective space $\PP^{g-1}$, so the real-analytic map is proper as well as finite-to-one. For other values of $k$, the source is the open stable locus.

We can also formulate this correspondence intrinsically. An integral-angle metric determines a flat unitary $\PP^1$-bundle $P\to X$ and a holomorphic section $s$ not locally constant in flat charts. The derivative of this section has zero divisor $R$. Conversely, pulling back the round fiber metric along such a section gives the metric on $X$. Isomorphisms of pairs are flat unitary bundle isomorphisms over the identity of $X$ carrying one section to the other. The correspondence between these pairs and metrics is bijective; hence, for fixed $R$, only finitely many pairs have noncoaxial monodromy.

If $V_s=s^*T_{P/X}$ is the vertical tangent line along the section, its derivative is a section of $K_X\otimes V_s$ and yields $K_X\otimes V_s\simeq\mathcal O_X(R)$. Identifying $\PSU(2)$ with $\mathrm{SO}(3)$, the section splits the associated oriented real rank-three bundle into its trivial real line and the underlying oriented real plane of $V_s$. Thus its second Stiefel--Whitney class has evaluation
\[
 \langle w_2,[X]\rangle=\deg V_s\pmod2=\deg R\pmod2.
\]
This gives the bundle interpretation of \eqref{eq:odd-lift}. A Hermitian--Einstein rank-two lift can have scalar curvature even when its projective connection is flat; odd ramification is compatible with such a projectively flat bundle.

 
\subsection{Meromorphic functions with prescribed ramification}\label{subsec:meromorphic}

We conclude with the case of trivial projective monodromy and an independent classical algebraic proof. Fix a complete effective ramification divisor $R$ on the compact curve $X$. Two nonconstant meromorphic maps $f_1,f_2:X\to\PP^1$ are equivalent when $f_2=M\circ f_1$ for a target M\"obius transformation $M$. A point of local degree $e_i$ contributes the coefficient $(1-e_i^2)/4$ to the $z^{-2}$ term of $\mathcal S(f)/2$, and the Schwarzian is holomorphic at every other point. It follows from Theorem~\ref{thm:unitary} and the equality of Schwarzians that there are only finitely many such equivalence classes.

We also give an independent proof using the Wronskian. Riemann--Hurwitz fixes the degree $d=\deg f$ by
\begin{equation}\label{eq:riemann-hurwitz}
 \deg R=2g-2+2d.
\end{equation}
The line bundle $L=f^*\mathcal O_{\PP^1}(1)$ satisfies
\begin{equation}\label{eq:ramification-root}
 L^{\otimes2}\simeq\mathcal O_X(R)\otimes K_X^{-1}.
\end{equation}
There are at most $2^{2g}$ such line bundles because two square roots differ by an element of $\Pic^0(X)[2]$, the group of degree-zero line bundles whose square is trivial \cite{For81}. For a fixed $L$, maps modulo target equivalence correspond to basepoint-free two-dimensional subspaces of $H^0(X,L)$.

For independent sections $s,t$ define their Wronskian by $W(s,t)=s\,dt-t\,ds$ in a local frame. The frame-derivative terms cancel, so $W(s,t)$ is a section of $K_X\otimes L^{\otimes2}$. It is nonzero: otherwise the meromorphic ratio $s/t$ would have zero derivative and be constant. The same conclusion holds for pencils with basepoints. Thus the linear Wronskian construction defines a morphism on the entire Grassmannian
\[
 \mathcal W_L:\Gr(2,H^0(X,L))\longrightarrow
 \PP H^0(X,K_X\otimes L^{\otimes2}).
\]
Its pullback of $\mathcal O(1)$ is the ample Pl\"ucker line bundle. A positive-dimensional fiber would contain a projective curve on which this line bundle has both positive degree, by ampleness, and zero degree, as the pullback from a point. Hence every fiber is zero dimensional and finite. Since two sections with the same divisor have constant ratio, the fixed divisor $R$ determines at most one projective section of $K_X\otimes L^{\otimes2}$. Combining the finitely many roots in \eqref{eq:ramification-root} with the finitely many pencils in each Wronskian fiber proves the assertion. If the degree in \eqref{eq:riemann-hurwitz} is not a positive integer, or no suitable pencil exists, the set is empty.

\begin{remark}[Rational functions]\label{rem:rational-functions}
For $X=\PP^1$, this is the familiar finite Wronski map for rational functions with a complete prescribed critical divisor \cite[Section~1]{EG01}\cite[Section~2.2]{Sch02}. Its degree is the Catalan number $\frac1d\binom{2d-2}{d-1}$, which bounds the number of distinct classes and gives the count for a general fiber. These sharper counts are classical. Degree-one maps give one class with no critical points. A map of degree at least two has at least two critical points; if exactly two occur, each has order $d-1$, and after sending them to $0,\infty$ the class is represented by $z^d$. Thus the cases of a sphere with few marked points agree with the projective-connection argument and the metric classification above.
\end{remark}

\bigskip
\noindent\textbf{AI disclosure.}
The author used ChatGPT 6-Astra to explore proof strategies, refine mathematical arguments, and assist with manuscript revisions. The author takes full responsibility for the mathematical content and the final text.

\end{document}